\documentclass[a4paper]{amsart}
\usepackage[all]{xy}

\usepackage{color}

\usepackage[colorlinks=true]{hyperref}
\usepackage{enumerate}
\usepackage{amssymb}
\usepackage{comment}
\usepackage{color}

\usepackage{verbatim}
\usepackage{amsmath,amscd,amsthm}

\newtheorem{thm}{Theorem}[section]

\newtheorem{prop}[thm]{Proposition}
\newtheorem{lem}[thm]{Lemma}
\newtheorem{cor}[thm]{Corollary}

\theoremstyle{definition}
\newtheorem{defn}[thm]{Definition}

\theoremstyle{remark}
\newtheorem{rem}[thm]{Remark}

\def\diff{\mathsf{d}}

\newcommand{\scdots}[2][]{\mathinner{#1\overset{#2}{\cdots}#1}}

\begin{document}


\title{Nijenhuis forms on Lie-infinity algebras associated to Lie algebroids}

\author{M. Jawad. Azimi}
\address{CMUC, Department of Mathematics, University of Coimbra, 3001-501 Coimbra, Portugal}
\email{mjazimi2004@gmail.com}
\author{C. Laurent-Gengoux}
\address{IECL, Universit\'e de Lorraine, 57045 Metz, France}
\email{camille.laurent-gengoux@univ-lorraine.fr}
\author{J. M. Nunes da Costa}
\address{CMUC, Department of Mathematics, University of Coimbra, 3001-501 Coimbra, Portugal}
\email{jmcosta@mat.uc.pt}

\begin{abstract}
Introducing Nijenhuis forms on $L_\infty$-algebras  gives a general frame to understand deformations of the latter. We give here a Nijenhuis interpretation of a deformation of an arbitrary Lie algebroid into an $L_\infty$-algebra. Then we show that Nijenhuis forms  on $L_\infty$-algebras also give a short and efficient manner to understand Poisson-Nijenhuis structures and, more generally, the so-called exact Poisson quasi-Nijenhuis structures with background.
\end{abstract}

\maketitle


\section*{Introduction}

\label{sec:introduction}

The present article is a continuation of \cite{AzimiLaurentCosta} by the same authors, where Nijenhuis forms on Lie-infinity algebras \cite{Lada-Sta} (from now on referred to as $L_\infty$-algebras) were introduced and various examples were given. Its purpose is to show that Nijenhuis forms on $L_\infty$-algebras are an efficient unification tool, which appears in contexts which are not obviously related neither between themselves, nor obviously related to Nijenhuis structures or higher structures.

Nijenhuis tensors arise naturally in complex geometry (a complex structure being a Nijenhuis tensor squaring to $-{\rm id}$), but they also appear while studying some  integrable systems, in particular those which are bihamiltonian, with one of the two Poisson structures being symplectic \cite{M78}. These so-called Poisson-Nijenhuis structures can be defined on an arbitrary Lie algebroid \cite{YKS, graburb}, give  hierachies of Poisson structures, and may even describe entirely the initial integrable system \cite{DamianouFernandes}. Several authors \cite{CGM, YKSBrazil, AntunesCosta, ALN13} have also studied Nijenhuis tensors on Loday algebras and Courant algebroids. These extensions share a common idea: a tensor $N$ is defined to be Nijenhuis for the structure $X$ (with $X$ $=$ Poisson, Lie algebroid, Courant algebroid and so on) when deforming $X$ twice by $N$ is like deforming $X$ by $N^2$. They also share a common feature: when $N$ is Nijenhuis, the structure $X_N$ obtained by deforming $X$ by $N$ is still of the same type as $X$ (i.e., $X_N$ is Poisson when $X$ is Poisson, $X_N$ is a Lie algebroid/Courant algebroid when $X$ is a Lie algebroid/Courant algebroid and so on) and $N$ is still Nijenhuis for $X_N$. This last property allows to repeat the procedure to get a hierarchy of structures of the same type as $X$, structures that can be shown to be compatible (i.e., any linear combinations of those are still of the same type as the structure $X$).

Interesting ``unifications" happen while generalizing the notion of Nijenhuis to Courant algebroids \cite{AntunesCosta}. For instance, a Poisson structure on a Lie algebroid $A$, being a tensor from $A^*$ to $A$, becomes itself a tensor on the Courant algebroid $A \oplus A^*$  which can be shown to be Nijenhuis. Hence, both Poisson structures and Nijenhuis tensors become, eventually, Nijenhuis tensors, and so do closed $2$-forms, $\Omega N$ and Poisson-Nijenhuis structures \cite{AntunesCosta}.

Having in mind that Courant algebroids are particular cases of $L_\infty$-algebras \cite{R98}, the authors of the present article have introduced in \cite{AzimiLaurentCosta} the notion of a Nijenhuis deformation on an arbitrary $L_\infty$-algebra, defined on a graded vector space $E$. To reach this level of generality, however, one has to accept Nijenhuis tensors which are not just tensors, i.e. linear endomorphisms, but collections of graded symmetric multilinear endomorphisms of $E$, exactly as the brackets of an $L_{\infty}$-algebra are a collection of $n$-ary brackets for an arbitrary $n$. These collections are referred to as (graded symmetric) vector valued forms and come equipped with a natural bracket: the Richardson-Nijenhuis bracket, which corresponds to the bracket of the coderivations of the graded symmetric algebra $S(E)$. $L_\infty$-algebras are precisely forms of degree $+1$ commuting with themselves.

Nijenhuis forms on $L_\infty$-algebras can however not be simply defined as vector valued forms ${\mathcal N}$ of degree $0$ such that deforming  the $L_\infty$-algebra twice by ${\mathcal N}$ (i.e., taking the Richardson-Nijenhuis bracket with ${\mathcal N}$) is like deforming it by the square of ${\mathcal N}$. The difficulty is that the square of a vector valued form can not be defined anymore. However, the equivalent of these squares are often quite natural, and \cite{AzimiLaurentCosta} has suggested a solution which consists in defining altogether Nijenhuis forms ${\mathcal N}$ and their ``square" $ {\mathcal K}$, with $ {\mathcal K}$ a vector valued form of degree $0$. By a square, we simply mean an other vector valued form such that deforming the $L_\infty$-structure twice by $ {\mathcal N}$ is like deforming it by ${\mathcal K}$. When this square ${\mathcal K}$ commutes with ${\mathcal N}$, a hierarchy of compatible $L_\infty$-algebras arises naturally by deforming by ${\mathcal N}$ several times \cite{YKS,ALN13}.
The purpose of this generalization, of course, is to obtain even more ``unifications" in the process.

The present article shows that two different kind of deformations can be explained by using Nijenhuis forms on $L_\infty$-algebras defined in \cite{AzimiLaurentCosta}.
The first one, detailed in Section \ref{section_Gerstenhaber_algebra}, is inspired by Delgado \cite{Delgado}, who gave an explicit construction showing that, for all Lie algebroid $A$, the graded space of sections of $\wedge A$ carries not only the Gerstenhaber bracket, but also an intriguing $L_\infty$-structure whose $2$-ary bracket is the Gerstenhaber bracket, the $3$-ary bracket is given by:
 $$ (P,Q,R) \mapsto [P,Q] \wedge R + \circlearrowleft_{P,Q,R}, \mbox{ for all $P,Q,R \in \Gamma(\wedge A)$}, $$
 and the $n$-ary brackets are given by similar formulas. We show in the present article that the wedge product of sections:
  $$ (P,Q) \mapsto P \wedge Q,  \mbox{ for all $P,Q \in \Gamma(\wedge A)$},$$
which is a graded symmetric 2-form of degree $0$ on $\Gamma(\wedge A)$, is co-boundary Nijenhuis (i.e., slightly more general than being Nijenhuis) with respect to the Gerstenhaber bracket. The square is, up to a coefficient, the  graded symmetric 3-form of degree $0$ given by the wedge product of three sections $(P,Q,R) \mapsto P \wedge Q \wedge R$.
 The reader used to higher structures maybe will not be surprised, but considering the wedge product as an analogous of a Nijenhuis tensor on a manifold may seem quite strange at first.
By (a slightly modified version of) the general theory of \cite{AzimiLaurentCosta}, a pencil
(indexed by ${\mathbb N}$) of $L_\infty$-algebras can be derived: the $L_\infty$-algebra of Delgado \cite{Delgado} is among this pencil.
In presence of a Nijenhuis tensor $N$, in the usual sense, of the Lie algebroid, the $L_\infty$-structures on this pencil can themselves be deformed by $N$.

In the meanwhile, we check that the wedge product of $n$-terms generate graded symmetric vector valued forms on $\Gamma(\wedge A)$ that, when equipped with the Richardson-Nijenhuis bracket, form a Lie algebra isomorphic to the Lie algebra of formal vector fields, which is an interesting point of its own.

Section \ref{section_PqNb} gives an other occurrence  of a Nijenhuis deformation of an $L_\infty$-algebra which is natural in Poisson geometry.
Poisson structures, Poisson-Nijenhuis, $\Omega N$ structures and $P\Omega$ structures and their counterpart with background can be all unified under the notion of exact Poisson quasi-Nijenhuis structures with background \cite{AntunesCosta}, a special case of Poisson quasi-Nijenhuis structures with background on Lie algebroids, introduced by Antunes in \cite{PauloAntunes}, which, in turn, are generalizations of Poisson quasi-Nijenhuis
 structures on manifolds \cite{Stienon&Xu} and on Lie algebroids \cite{RaquelPauloJoana}. Exact Poisson quasi-Nijenhuis structures with background are made of a bivector $\pi$, a 2-form $\omega$ and a Nijenhuis tensor $N$ on a Lie algebroid $A$ assumed to satisfy several relations, in relation with some $3$-form (the background).  We show that these relations precisely mean that the only natural vector valued form of degree $0$ that can be constructed on the graded vector space of sections of $\wedge A$ out of $\pi, \omega$ and $N$ is a co-boundary Nijenhuis form with a certain square (which is again the only natural one) for the $L_\infty$-algebra of sections of $\wedge A$ (equipped with some background $3$-form).
We also explain how to recover the notion of Poisson quasi-Nijenhuis manifolds of Sti\'enon and Xu \cite{Stienon&Xu} in terms of Nijenhuis forms on $L_\infty$-algebras.


\section{Richardson-Nijenhuis bracket }
\subsection*{Richardson-Nijenhuis bracket and $L_\infty$-algebras}
In this section we recall from \cite{AzimiLaurentCosta} the notion of Richardson-Nijenhuis bracket of graded symmetric vector valued
forms on graded vector spaces and the notion of $L_\infty$-algebra. We also recall that it characterizes $L_\infty$-structures on graded vector spaces.
We start by fixing some notations on graded vector spaces.

Let $E=\oplus_{i \in {\mathbb Z}} E_i$ be a graded vector space over a field $\mathbb{K}=\mathbb{R}$ or $\mathbb{C}$.
For a given $i\in \mathbb{Z}$, the vector space $E_i$ is called the component of degree $i$,
elements of $E_i$ are called \emph{homogeneous elements of degree $i$}, and elements in the union $\cup_{i\in \mathbb{Z}} E_i$
are called the \emph{homogeneous elements}.
We denote by $|X|$ the degree of a non-zero homogeneous element $X$. Given a graded vector space $E=\oplus_{i\in \mathbb{Z}}E_i$ and an integer $p$,
 $E[p]$ is the graded vector space whose  component of degree $i$ is $E_{i+p}$.

  We denote by $S(E)$ the symmetric space of $E$.
For a given $k\geq 0$, $S^k(E)$ is the image of $\otimes^{k} E$ through the quotient map
$\otimes E \to S(E)$. One has the following natural
 decomposition
   \begin{equation*} S(E) = \oplus_{k \geq 0} S^k (E), \end{equation*}
where $S^0(E)$ is simply the field $\mathbb{K}$ when there exists an $i_0 \in \mathbb Z$ such that $E_i=0$ for all $i\leq i_0$ (or all $i\geq i_0$). We denote by $\tilde{S}(E)$ the completion of $S(E)$, the algebra of
 formal sums of the elements in $S(E)$.
  Also, when all the components in the graded space $E$ are of finite dimension, the dual of $S^k(E)$ is isomorphic to $S^k(E^*)$, for all $k\geq 0$. In this case,  there is a one to one correspondence between
\begin{enumerate}
\item[(i)] graded symmetric $k$-linear maps on the graded vector space $E$,
\item[(ii)] linear maps from the space $S^k(E)$ to $E$,
 \item[(iii)] $S^k(E^*)\otimes E$.
\end{enumerate}
Elements of the space $S^k(E^*)\otimes E$ are called \emph{symmetric vector valued  $k$-forms}. Notice that $S^0(E^*)\otimes E$, the space of vector valued zero-forms, is isomorphic to the space $E$.


Let $E$ be a graded vector space, $E= \oplus_{i\in \mathbb{Z}}E_i$. \emph{The insertion operator} of a symmetric vector valued $k$-form $K$ is an operator
$$ \iota_K : S(E^*)\otimes E \to S(E^*)\otimes E$$
defined by
\begin{equation*}\label{equ:insertion}
\iota_{K}L(X_1,...,X_{k+l-1})=\sum_{\sigma \in Sh(k,l-1)}\epsilon(\sigma)L(K(X_{\sigma(1)},...,X_{\sigma(k)}),...,X_{\sigma(k+l-1)}),
\end{equation*}
for all $L \in S^{l}(E^*)\otimes E $, $l\geq 0$ and $X_1,\cdots, X_{k+l-1} \in E$, where $Sh(i,j-1)$ stands for the set of $(i,j-1)$-unshuffles and
$\epsilon(\sigma)$ is the Koszul sign which
is defined as follows
   \begin{equation*}
   X_{\sigma(1)}\odot\cdots \odot X_{\sigma(n)}=\epsilon(\sigma)X_{1}\odot\cdots \odot X_{n},
   \end{equation*}
for all $X_{1}, \cdots, X_{n} \in E$, with $\odot$ the symmetric product.

Now, we define a bracket on the space $\tilde{S}(E^*) \otimes E$ as follows.
Given a symmetric vector valued $k$-form $K\in S^{k}(E^*)\otimes E$ and a symmetric vector valued $l$-form $L\in S^{l}(E^*)\otimes E$, the {\em Richardson-Nijenhuis bracket} of $K$ and $L$ is the symmetric vector valued $(k+l-1)$-form $[K,L]_{_{RN}}$, given by
\begin{equation*} \label{RNbracket}
[K,L]_{_{RN}}=\iota_{K}L-(-1)^{\bar{K}\bar{L}}\iota_{L}K,
\end{equation*}
where $\bar{K}$ is the degree of $K$ as a graded map, that is $K(X_1,\cdots, X_k)\in E_{1+\cdots+k+\bar{K}},$ for all $X_i \in E_i$.
For an element $X \in E$, $\bar{X}=|X|$, that is, the degree of a vector valued $0$-form, as a graded map, is just its degree as an element of $E$.
If $K\in S^k(E^*)\otimes E$ is a vector valued $k$-form, an easy computation gives
\begin{equation} \label{kform}
K(X_1,\cdots,X_k)=[X_k,\cdots,[X_2,[X_1,K]_{_{RN}}]_{_{RN}}\cdots]_{_{RN}},
\end{equation}
for all $X_1,\cdots,X_k\in E$

Next, we recall the notion of $L_\infty$-algebra, following the conventions of \cite{Getzler}.
\begin{defn}\label{def:SymLinfty}
An {\em $L_{\infty}$-algebra} is a graded vector space $E= \oplus_{i\in \mathbb{Z}}E_i$ together with a family of symmetric vector valued forms $(l_i)_{i\geq 1}$ of degree $1$, with $l_i: \otimes^i E\to E$ satisfying the following relation:
\begin{equation*}\label{L_infty-def}
\sum_{i+j=n+1}\sum_{\sigma \in Sh(i,j-1)}\epsilon(\sigma)l_j(l_i(X_{\sigma(1)},\cdots,X_{\sigma(i)}),\cdots,X_{\sigma(n)})=0,
\end{equation*}
for all $n\geq 1$ and all homogeneous  $X_1,\cdots, X_n \in E$, where $\epsilon(\sigma)$ is the Koszul sign. The family of symmetric vector valued forms $(l_{i})_{i\geq 1}$ is called an {\em $L_{\infty}$-structure} on the graded vector space $E$.
Usually, we denote this $L_{\infty}$-structure by $\mu:=\sum_{i\geq 1} l_i$.
\end{defn}
Two  $L_\infty$-structures $\mu$ and $\mu'$ on a graded vector space are said to be \emph{compatible} if $[\mu,\mu']_{_{RN}}=0$.
A \emph{pencil} indexed by a set $I$ is a family $(\mu_i)_{i \in I}$ of pairwise compatible $L_\infty$-structures.

\

The Richardson-Nijenhuis bracket on graded vector spaces, introduced previously, is intimately related to $L_{\infty}$-algebras. In the next theorem, that appears in an implicit form in \cite{Roytenberg}, we use the Richardson-Nijenhuis bracket to characterize a $L_{\infty}$-structure on a graded vector space $E$ as a homological vector field on $E$.

\begin{thm}\label{th:linftyRN}
Let $E=\oplus_{i\in \mathbb{Z}}E_i$ be a graded vector space and
 $(l_i)_{i \geq 1}: \otimes^{i}E \to E$ be a family of symmetric vector valued forms on $E$ of degree $1$.
 Set $\mu= \sum_{i\geq 1} l_i$. Then,
 $\mu$ is an $L_{\infty}$-structure on $E$ if and only if $[\mu,\mu]_{_{RN}}=0$.
\end{thm}



\subsection*{Richardson-Nijenhuis bracket and Gerstenhaber algebras}

Let $(A= \oplus_{i\in \mathbb{Z}} A^{i}, \wedge)$ be a graded commutative associative algebra. Recall
 that a graded Lie algebra structure on
$A\left[1\right]$
is
called a \emph{Gerstenhaber algebra} bracket if, for each $P\in A^{i}\left[1\right], \left[P,\cdot \right] $ is a derivation
of degree $i$ of $(A= \oplus_{i\in \mathbb{Z}} A^{i}, \wedge)$. A Gerstenhaber algebra will be denoted by $(A, \left[\cdot,\cdot \right],\wedge)$.
\begin{rem}\label{Gerstenhaberalgebra}
It follows from the definition of a Gerstenhaber algebra $(A,\left[\cdot,\cdot \right], \wedge)$ that, for all homogeneous elements $P, Q, R \in A$, the following identities hold:
\begin{equation} \label{graded_skew}
\left[P,Q \right]=-\left(-1\right)^{\left(|P|-1\right)\left(|Q|-1\right)}\left[Q,P \right],
\end{equation}
 \begin{equation} \label{graded_Leibniz_rule}
\left[P,Q\wedge R \right]=\left[P,Q \right]\wedge R+\left(-1\right)^{\left(|P|-1\right)|Q|}Q\wedge \left[P, R \right].
\end{equation}
\end{rem}

Let $\left(A, \left[\cdot,\cdot \right], \wedge\right)$ be a Gerstenhaber algebra. For all positive integers $k$, and all homogeneous elements $P_1, ..., P_k \in  A$, set
\begin{equation*}
E_{-k}:=A^{k},  \quad \textrm{i.e.}, \quad |P|=-k, \quad \textrm{if} \quad P \in A^{k},
\end{equation*}
\begin{equation*}
\mathcal{N}_{k}\left(P_1,\cdots,P_k\right):=P_1\wedge\dots\wedge P_k,
\end{equation*}
\begin{equation*}
l_1:=0, \quad \,\,\,\,\,\,l_2(P_1,P_2):=(-1)^{|P_1|}\left[P_1,P_2\right]
\end{equation*}
 and for $i>2,$
 $$\begin{array}{rcl}
l_i\left(P_1,\cdots, P_i\right)&:=&\iota_{l_2}\mathcal{N}_{i-1}\left(P_1,\cdots, P_i\right) \\
                               &=&\!\!\!\sum_{\sigma\in Sh\left(2,i-2\right)}\epsilon\left(\sigma\right)\left(-1\right)^{|P_{\sigma\left(1\right)}|}
\left[P_{\sigma\left(1\right)}, P_{\sigma\left(2\right)}\right]\wedge P_{\sigma\left(3\right)}\wedge \cdots \wedge P_{\sigma\left(i\right)}.\\
\end{array}
$$
By construction, $\mathcal{N}_{k} \in S^k (E^*) \otimes E$ has degree $0$ and $l_k \in  S^{k} (E^*) \otimes E$ has degree $1$, for all $k \geq 1$.



\begin{lem} \label{mainlemma}For any Gerstenhaber algebra $\left(A, \left[\cdot,\cdot \right], \wedge\right)$, the  $l_i$'s and $\mathcal{N}_i$'s introduced above satisfy:
\begin{enumerate}
\item [(a)] $l_2(P\wedge Q, R)= \left(-1\right)^{|Q||R|}l_2(P, R)\wedge Q+\left(-1\right)^{|P|\left(|Q|+|R|\right)}l_2(Q, R)\wedge P$, for all homogeneous $P,Q,R \in A$;
\item [(b)] $\left[\mathcal{N}_{i},\mathcal{N}_{j}\right]_{_{RN}}=\frac{\left(j-i\right)\left(i+j-1\right)!}{i!j!}\mathcal{N}_{i+j-1}$, for all positive integers $i,j$;
\item [(c)] $\left[\mathcal{N}_{m}, l_{n}\right]_{_{RN}}=\binom {m+n-2}{m}l_{m+n-1}$, for all $m,n\geq 2$;
\item [(d)] $\left[l_m, l_n\right]_{_{RN}}=0$, for all positive integers $m,n$.
\end{enumerate}
\end{lem}
\begin{proof}
Equations (\ref{graded_skew}) and (\ref{graded_Leibniz_rule}) in Remark \ref{Gerstenhaberalgebra} and the definition of $l_2$ prove $\textit{(a)}$.
A direct computation proves $\textit{(b)}$.
Let us now prove $\textit{(c)}$. By definition of the insertion operator, we have
\begin{equation}\label{iNilk}
\begin{array}{l}
\iota_{_{\mathcal{N}_{m}}}l_n(P_1,\cdots,P_{m+n-1})=\\
\sum_{\sigma \in Sh\left(m,n-1\right)} \epsilon \left(\sigma\right) l_n\left(P_{\sigma\left(1\right)}\wedge\cdots\wedge P_{\sigma\left(m\right)}
,P_{\sigma\left(m+1\right)},\cdots, P_{\sigma\left(m+n-1\right)}\right).
\end{array}
\end{equation}
For a fixed $\sigma \in Sh\left(m,n-1\right)$ we have, by definition of $l_n$,
\begin{equation*}\label{iotaNmln1}
\begin{array}{rcl}
&l_n\left(P_{\sigma\left(1\right)}\wedge\cdots\wedge P_{\sigma\left(m\right)},P_{\sigma\left(m+1\right)},\cdots, P_{\sigma\left(m+n-1\right)}\right)&\\
=&\sum_{i=m+1}^{m+n-1}\left(-1\right)^{\alpha_1}l_2\left(P_{\sigma\left(1\right)}\wedge\cdots\wedge P_{\sigma\left(m\right)}, P_{\sigma\left(i\right)}\right)
\wedge P_{\sigma\left(m+1\right)}\wedge\scdots{\widehat{i}} \wedge P_{\sigma\left(m+n-1\right)}&\\
+&\sum_{m+1\leq r<s\leq m+n-1}\left(-1\right)^{\alpha_2}l_2\left(P_{\sigma\left(r\right)}, P_{\sigma\left(s\right)}\right)
\wedge P_{\sigma\left(1\right)}\wedge\scdots{\widehat{r,s}}\wedge P_{\sigma\left(m+n-1\right)},
\end{array}
\end{equation*}
where \begin{equation*}
\alpha_1=|P_{\sigma\left(i\right)}|\left(|P_{\sigma\left(m+1\right)}|+\cdots+|P_{\sigma\left(i-1\right)}|\right)
\end{equation*}
and
 \begin{equation*}
  \alpha_2=|P_{\sigma\left(s\right)}|\left(|P_{\sigma\left(1\right)}|+\cdots+|P_{\sigma\left(s-1\right)}|\right)
  +|P_{\sigma\left(r\right)}|\left(|P_{\sigma\left(s\right)}|+|P_{\sigma\left(1\right)}|+\cdots+|P_{\sigma\left(r-1\right)}|\right).
  \end{equation*}
Hence, using (\ref{graded_skew}) and (\ref{graded_Leibniz_rule}), we get
\begin{equation}\label{iotaNmln2}
\begin{array}{rcl}
&l_n\left(P_{\sigma\left(1\right)}\wedge\cdots\wedge P_{\sigma\left(m\right)},P_{\sigma\left(m+1\right)},\cdots, P_{\sigma\left(m+n-1\right)}\right)&\\
=&\sum_{i=m+1}^{m+n-1}\left(-1\right)^{\alpha_1}\sum_{j=1}^{m}\left(-1\right)^{\alpha_3}l_2\left(P_{\sigma\left(j\right)}, P_{\sigma\left(i\right)}\right)
\wedge P_{\sigma\left(1\right)}\wedge\scdots{\widehat{j,i}} \wedge P_{\sigma\left(m+n-1\right)}&\\
+&\sum_{m+1\leq r<s\leq m+n-1}\left(-1\right)^{\alpha_2}l_2\left(P_{\sigma\left(r\right)}, P_{\sigma\left(s\right)}\right)
\wedge P_{\sigma\left(1\right)}\wedge\scdots{\widehat{r,s}}\wedge P_{\sigma\left(m+n-1\right)},&\\
\end{array}
\end{equation}
where
\begin{equation*}
  \alpha_3=|P_{\sigma\left(i\right)}|\left(|P_{\sigma\left(1\right)}|+\cdots+|P_{\sigma\left(i-1\right)}|\right)+|P_{\sigma\left(j\right)}
  |\left(|P_{\sigma\left(1\right)}|+\cdots+|P_{\sigma\left(j-1\right)}|+|P_{\sigma\left(i\right)}|\right).
  \end{equation*}
Now, for two integers $i$ and $j$, with $m\leq i\leq m+n-1$ and $1\leq j\leq m$, we introduce a new permutation, which depends on $\sigma, i$ and $j$, denoted by
$\tau_{\sigma ij}$, as follows:

\begin{equation*}\label{sigmaij}
\sigma\left(k\right) = \begin{cases}
\tau_{\sigma ij} \left(k+2\right),&\mbox{if }  1\leq k < j \\
\tau_{\sigma ij} \left(1\right),&\mbox{if }  k=j \\
\tau_{\sigma ij} \left(k+1\right),&\mbox{if }  j < k < i \\
\tau_{\sigma ij} \left(2\right),&\mbox{if }  k=i \\
\tau_{\sigma ij} \left(k\right),&\mbox{if }  i < k < m+n-1,
 \end{cases}
\end{equation*}
when $\sigma\left(j\right)<\sigma\left(i\right)$ and by
\begin{equation*}\label{sigmaij}
\sigma\left(k\right) = \begin{cases}
\tau_{\sigma ij} \left(k+2\right),&\mbox{if }  1\leq k < j \\
\tau_{\sigma ij} \left(2\right),&\mbox{if }  k=j \\
\tau_{\sigma ij} \left(k+1\right),&\mbox{if }  j < k < i \\
\tau_{\sigma ij} \left(1\right),&\mbox{if }  k=i \\
\tau_{\sigma ij} \left(k\right),&\mbox{if }  i < k < m+n-1,
 \end{cases}
\end{equation*}
when $\sigma\left(i\right)<\sigma\left(j\right)$.
Next, for two integers $r$ and $s$, with $m+1\leq r < s \leq m+n-1$, we introduce another permutation, which depends on $\sigma, r$ and $s$, denoted by
$\tau_{\sigma rs}$, as follows:

\begin{equation*}\label{sigmars}
\sigma\left(k\right) = \begin{cases}
\tau_{\sigma rs} \left(k+2\right),&\mbox{if }  1\leq k < r \\
\tau_{\sigma rs} \left(1\right),&\mbox{if }  k=r \\
\tau_{\sigma rs} \left(k+1\right),&\mbox{if }  j < k < s \\
\tau_{\sigma rs} \left(2\right),&\mbox{if }  k=s \\
\tau_{\sigma rs} \left(k\right),&\mbox{if }  s < k < m+n-1.
 \end{cases}
\end{equation*}
Then, we have \begin{equation*}\tau_{\sigma ij},\tau_{\sigma rs}\in Sh\left(2,m+n-3\right),\end{equation*}
\begin{equation*}\epsilon\left(\sigma\right)\left(-1\right)^{\alpha_1}\left(-1\right)^{\alpha_3}=\epsilon\left(\tau_{\sigma ij}\right),\end{equation*}
\begin{equation*}\epsilon\left(\sigma\right)\left(-1\right)^{\alpha_2}=\epsilon\left(\tau_{\sigma rs}\right)\end{equation*}
and
\begin{equation}\label{iotaNmln3}
\begin{array}{rcl}
&\epsilon\left(\sigma\right)l_n\left(P_{\sigma\left(1\right)}\wedge\cdots\wedge P_{\sigma\left(m\right)},P_{\sigma\left(m+1\right)}\cdots, P_{\sigma\left(m+n-1\right)}\right)&\\
=&\sum_{i=m+1}^{m+n-1}\sum_{j=1}^{m}\epsilon\left(\tau_{\sigma ij}\right)l_2\left(P_{\tau_{\sigma ij}\left(1\right)}, P_{\tau_{\sigma ij}\left(2\right)}\right)
\wedge P_{\tau_{\sigma ij}\left(3\right)}\wedge\cdots \wedge P_{\tau_{\sigma ij}\left(m+n-1\right)}&\\
+&\sum_{m+1\leq r<s\leq m+n-1}\epsilon\left(\tau_{\sigma rs}\right)l_2\left(P_{\tau_{\sigma rs}\left(1\right)}, P_{\tau_{\sigma rs}\left(2\right)}\right)
\wedge P_{\tau_{\sigma rs}\left(3\right)}\wedge\cdots\wedge P_{\tau_{\sigma rs}\left(m+n-1\right)}.&\\
\end{array}
\end{equation}
Equation (\ref{iotaNmln3}) shows that for each $\sigma \in Sh\left(m, n-1\right)$,
\begin{equation*}
\epsilon\left(\sigma\right)l_n\left(P_{\sigma\left(1\right)}\wedge\cdots\wedge P_{\sigma\left(m\right)},P_{\sigma\left(m+1\right)}\cdots, P_{\sigma\left(m+n-1\right)}\right)
\end{equation*}
has $m\left(n-1\right)+\binom {n-1}{2}$ terms of the form
\begin{equation}\label{form1}
\epsilon \left(\tau\right) l_2\left(P_{\tau\left(1\right)},P_{\tau\left(2\right)}\right)\wedge P_{\tau\left(3\right)}\cdots, P_{\tau\left(i+k-1\right)},
\end{equation}
with $\tau \in Sh\left(2,m+n-1\right)$. \footnote{Note that not all the permutations in $Sh\left(2,m+n-1\right)$ appear in (\ref{form1}) for
a single $\sigma$. But this happens when $\sigma$ varies in $Sh\left(m,n-1\right)$ and in this case each permutation in $Sh\left(2,m+n-1\right)$ repeats by a same  number.}
 Therefore the right hand side of (\ref{iNilk}) has $A=\binom {m+n-1}{m}\left(m\left(n-1\right)+\binom {n-1}{2}\right)$ terms of the form (\ref{form1}). Note that
\begin{equation*}
\begin{array}{rcl}
A&=&\binom {m+n-1}{m}\left(\frac{\left(n-1\right)\left(n+2m-2\right)}{2}\right)\\
 &=&\frac{\left(m+n-1\right)!}{2!\left(m+n-3\right)!}\times \frac{\left(m+n-3\right)!}{m!\left(n-2\right)!}\times \left(n+m-2+m\right)\\
 &=&\binom {m+n-1}{2}\times\left(\binom {m+n-2}{m}+\binom {m+n-3}{m-1}\right).
\end{array}
\end{equation*}
This shows that
\begin{equation}\label{conclusion1}
\iota_{_{\mathcal{N}_{m}}}l_n=\left(\binom {m+n-2}{m}+\binom {m+n-3}{m-1}\right)l_{m+n-1}.
\end{equation}
Similar computations show that
\begin{equation}\label{conclusion2}
\iota_{_{l_n}}\mathcal{N}_{m}=\binom {m+n-3}{m-1}l_{m+n-1}.
\end{equation}
Equations (\ref{conclusion1}) and (\ref{conclusion2}) prove $\textit{(c)}$.

We now prove $\textit{(d)}$. Note that the proof below will be interpreted easily with the help of co-boundary Nijenhuis forms that
appear in the next section.

Item $\textit{(c)}$ and the graded Jacobi identity of the Richardson-Nijenhuis yield:
\begin{equation*}\label{lilj=0}
\begin{array}{rcl}
\left[l_{m}, l_n \right]_{_{RN}}&=&\left[\left[\mathcal{N}_{m-1}, l_2 \right]_{_{RN}},\left[\mathcal{N}_{n-1}, l_2 \right]_{_{RN}} \right]_{_{RN}}\\
                                &=&\left[\mathcal{N}_{n-1}, \left[l_2,\left[\mathcal{N}_{m-1}, l_2 \right]_{_{RN}} \right]_{_{RN}} \right]_{_{RN}}+
                                \left[l_2, \left[\left[\mathcal{N}_{m-1}, l_2 \right]_{_{RN}},\mathcal{N}_{n-1} \right]_{_{RN}} \right]_{_{RN}}\\
                                &=&\left[l_{2},\left[l_{m}, \mathcal{N}_{n-1}\right]_{_{RN}}  \right]_{_{RN}}\\
                                &=& -\binom {m+n-3}{n-1}\left[l_{2}, l_{m+n-2} \right]_{_{RN}} \\
                                &=& -\binom {m+n-3}{n-1}\left[l_{2},\left[\mathcal{N}_{m+n-3}, l_2 \right]_{_{RN}}\right]_{_{RN}}\\
                                &=& 0.

\end{array}
\end{equation*}
\end{proof}

\begin{rem}
The Lie algebra $\mathcal V$ of polynomial vector fields on $\mathbb R$ is generated by the vectors
\begin{equation*}
\nu_i= x^i \frac{\partial}{\partial x}, \quad i \geq 0,
\end{equation*}
whose Lie bracket is given by $[\nu_i, \nu_j]=(j-i) \nu_{i+j-1}$. It acts on the space $\mathcal P$ of polynomials, which is generated by $x_j:= x^j$, as
\begin{equation*}
\nu_i[x_j]= j \, x_{i+j-1},
\end{equation*}
allowing to endow $ {\mathcal V} \ltimes {\mathcal P}$ with a Lie algebra structure. Lemma \ref{mainlemma} means that, when equipped with the Schouten-Nijenhuis bracket, the vector space generated by $(\mathcal{N}_{k})_{k \geq 1}$ and $(l_k)_{k \geq 1}$ is isomorphic to $ {\mathcal V} \ltimes {\mathcal P}$ through the isomorphism:
$$
\begin{array}{lcl}
\nu_i & \mapsto & i! \, \mathcal{N}_{i}\\
x_i & \mapsto & i! \, l_{i+1}.
\end{array}
$$
\end{rem}

\

We conclude this section with the following proposition:
\begin{prop}\label{maintheorem}
Let $\left(A=\oplus_{i\in \mathbb{Z}}A^{i}, \left[\cdot,\cdot \right], \wedge\right)$ be a Gerstenhaber algebra. Then, $l_2$ given by
\begin{equation} \label{definition_l2}
l_2(P_1,P_2):=(-1)^{|P_1|}\left[P_1,P_2\right],
\end{equation}
for all homogeneous elements $P_1, P_2 \in A$,
is an $L_{\infty}$-structure on the graded vector space $E=\oplus_{i \in \mathbb Z}E_{k}$, with $E_{-i}:= A^{i}$.
\end{prop}

\section{Pencils of $L_{\infty}$-structures on Lie algebroids} \label{section_Gerstenhaber_algebra}
In this section we recall from \cite{AzimiLaurentCosta} the notion of Nijenhuis vector valued form with respect to a given vector valued form $\mu$
 and deformation of $\mu$ by a Nijenhuis vector valued form. We then describe two examples.
\begin{defn} \label{def:Nijenhuis}
Let $E$ be a graded vector space and $\mu$ be a symmetric vector valued form on $E$ of degree $1$. A vector valued form ${\mathcal N}$ of degree zero is called
\begin{itemize}
  \item {\em weak Nijenhuis} with respect to $\mu$ if
       \begin{equation*}\label{weakN}
       \left[\mu,\left[{\mathcal N},\left[{\mathcal N},\mu\right]_{_{RN}}\right]_{_{RN}}\right]_{_{RN}}=0,
       \end{equation*}
  \item {\em co-boundary Nijenhuis} with respect to $\mu$ if there exists a vector valued form ${\mathcal K}$ of degree zero, such that
   \begin{equation*}\label{boundN}
       \left[{\mathcal N},\left[{\mathcal N},\mu\right]_{_{RN}}\right]_{_{RN}}=\left[{\mathcal K},\mu\right]_{_{RN}},
       \end{equation*}
  \item {\em Nijenhuis} with respect to $\mu$  if there exists a vector valued form ${\mathcal K}$ of degree zero, such that
       \begin{equation*}\label{strongN}
       \left[{\mathcal N},\left[{\mathcal N},\mu\right]_{_{RN}}\right]_{_{RN}}=\left[{\mathcal K},\mu\right]_{_{RN}}\,\,\,\hbox{and }\,\,\,\left[{\mathcal N},{\mathcal K}\right]_{_{RN}}=0.
       \end{equation*}
  Such a ${\mathcal K} $ is called a {\em square} of ${\mathcal N} $.
\end{itemize}
  Notice that
  ${\mathcal N}$ may contain an element of the underlying graded vector space.
\end{defn}
Recall from \cite{AzimiLaurentCosta} the following proposition:
\begin{prop}
\label{prop:Otherpaper} Let $ {\mathcal N}$ be a Nijenhuis form in any of the senses above for a $L_\infty $-structure $\mu$, then $\left[{\mathcal N},\mu\right]_{_{RN}}$ is an $L_\infty$-structure compatible with $\mu $.
If $ {\mathcal N}$ is weak Nijenhuis with respect to $\mu$, the converse also holds.
\end{prop}
The $L_\infty $-structure $\left[{\mathcal N},\mu\right]_{_{RN}}$ is  called  the \emph{deformed structure} of $\mu$  by $ {\mathcal N}$.



\subsection*{$L_\infty$-structures mixing products and brackets on Gerstenhaber algebras}

Next we give an example of co-boundary Nijenhuis vector valued forms with respect to $L_{\infty}$-algebras obtained on a Gerstenhaber algebra.

\begin{thm}\label{Theorem}
Let $\left(A=\oplus_{k\in \mathbb{Z}}A^{k}, \left[\cdot,\cdot \right], \wedge\right)$ be a Gerstenhaber algebra and set $E=\oplus_{k\in \mathbb{Z}}E^{k}$, with $E_{-k}:= A^{k}$.
  For all positive integers $i$ and all homogeneous elements $P_1,\cdots,P_i \in A $, set
 $\mathcal{N}_{i}\left(P_1,\cdots,P_i\right):=P_1\wedge\dots\wedge P_i$.
  Let
$l_1:=0$, \begin{equation*}  l_2(P_1,P_2):=(-1)^{|P_1|}\left[P_1,P_2\right]\end{equation*}
 and for $i>2$,
 \begin{equation*}
\begin{array}{rcl}
l_i\left(P_1,\cdots, P_i\right)&:=&\iota_{l_2}\mathcal{N}_{i-1}\left(P_1,\cdots, P_i\right) \\
                               &=&\!\!\!\sum_{\sigma\in Sh\left(2,i-2\right)}\epsilon\left(\sigma\right)\left(-1\right)^{|P_{\sigma\left(1\right)}|}
\left[P_{\sigma\left(1\right)}, P_{\sigma\left(2\right)}\right]\wedge P_{\sigma\left(3\right)}\wedge \cdots \wedge P_{\sigma\left(i\right)}.\\
\end{array}
\end{equation*}
 Then,
\begin{enumerate}
\item  [(a)] for any $n \geq 1$, $ l_n$ is an $L_\infty$-structure on $E$ and for all $n, m \geq 1$, $\mathcal{N}_m$ is co-boundary Nijenhuis with respect to $l_n$ and, up to a constant, $ {l}_{n+m-1} $ is the deformed structure of $l_m$ by $ {\mathcal N}_{m} $;
\item  [(b)] the family $(l_n)_{n \geq 1}$ is a pencil of $L_\infty$-algebras on $E$, therefore $\mu=\sum_{i\geq 1}a_i l_{i}$ is an $L_\infty$-algebra on $E$, for all reals $a_i$;
\item  [(c)] $\mathcal{N}:=\sum_{i\geq 1}b_i\mathcal{N}_{i}$, where $b_i\in \mathbb{R}$ for all positive integer $i$,
 is a co-boundary Nijenhuis vector valued form with respect to $l_n$;
\item [(d)] $\mathcal{N}$ is a weak Nijenhuis vector valued form with respect to  $\mu=\sum_{i\geq 1} a_il_i$, with $a_i \in \mathbb R$.
\end{enumerate}
\end{thm}
\begin{proof}
By item $\textit{(c)}$ of  Lemma \ref{mainlemma}, for all $m,n \geq 1$ there exists a constant $A_{n,m}$ such that
\begin{equation} \label{N_m_coboundary}
 \begin{array}{rcl}
 \left[\mathcal{N}_m, \left[\mathcal{N}_m, l_n \right]_{_{RN}}\right]_{_{RN}} &=&  A_{n,m}
 \left[\mathcal{N}_{2m-1}, l_n \right]_{_{RN}}.
 \end{array}
\end{equation}
This formula means that  ${\mathcal N}_m$ is co-boundary Nijenhuis for $l_2$ for all $ m\geq 1$.
 By item $\textit{(c)}$ of  Lemma \ref{mainlemma} specialized to $n=2$, the deformed structure is, up to a constant, $l_{m+1}$.
 The latter is therefore by Proposition \ref{prop:Otherpaper} an $L_\infty$-structure for all $m \geq 1$, and is compatible with $l_2$.
 Relation (\ref{N_m_coboundary}) then also means that
  ${\mathcal N}_m$ is co-boundary Nijenhuis with respect to $l_n$ for all $ m,n \geq 1$.
 By item $\textit{(c)}$ of  Lemma \ref{mainlemma}, the deformed structure is, up to a constant, $l_{n+m-1}$.
By item $\textit{(d)}$ of Lemma \ref{mainlemma},
$l_n$ and $l_m$  are compatible for all $m,n \geq 1$, i.e. $(l_n)_{n \geq 1}$ is a pencil of $L_\infty$-algebras. This proves the two first items.

 Observe that, using Equation (\ref{conclusion2}), we have
 \begin{equation*}\label{Nijenhuis}
 \left[\mathcal{N}, l_n \right]_{_{RN}}= \left[\sum_{i\geq 1}b_i\mathcal{N}_{i}, l_n \right]_{_{RN}}= \sum_{i\geq 1} b_i \binom{i+n-2}{i} l_{i+n-1}.
 \end{equation*}
 Hence,
 \begin{equation*}\label{Nijenhuis}
 \begin{array}{rcl}
 \left[\mathcal{N}, \left[\mathcal{N}, l_n \right]_{_{RN}}\right]_{_{RN}}
 &=&\!\!\sum_{i\geq 1} b_i \binom{i+n-2}{i} \left[\sum_{j\geq 1} b_j\mathcal{N}_{j}, l_{n+i-1} \right]_{_{RN}}
  \\\\
 &=&\!\! \sum_{i,j\geq 1}b_ib_j\binom{i+n-2}{i}\binom{j+i+n-3}{j}l_{n+i+j-2}\\\\
 &=&\!\! \sum_{i,j\geq 1}b_ib_j \frac{\binom{i+n-2}{i}\binom{j+i+n-1}{j}}{\binom{j+i+n-3}{i+j-1}} \left[\mathcal{N}_{i+j-1}, l_n \right]_{_{RN}}\\\\
 &=&\!\! \left[\sum_{i,j\geq 1}b_ib_j \frac{\binom{i+n-2}{i}\binom{j+i+n-3}{j}}{\binom{j+i+n-3}{i+j-1}} \mathcal{N}_{i+j-1}, l_n \right]_{_{RN}}.
 \end{array}
 \end{equation*}
 This proves that $\mathcal{N}$ is co-boundary Nijenhuis with respect to the $L_{\infty}$-structure $l_n$, with square
 \begin{equation*}
 \sum_{i,j\geq 1}b_ib_j \frac{\binom{i+n-2}{i}\binom{j+i+n-3}{j}}{\binom{j+i+n-3}{i+j-1}} \mathcal{N}_{i+j-1},
 \end{equation*}
 which proves $\textit{(c)}$. By item $\textit{(c)}$ of  Lemma \ref{mainlemma}, there is a family of reals $(c_i)_{i\geq 1}$ such that  $\left[\mathcal{N}, \mu \right]_{_{RN}}=\sum_{i\geq 1}c_il_i$. From
 item $\textit{(b)}$,  $\left[\mathcal{N}, \mu \right]_{_{RN}}$ is an $L_{\infty}$-structure on the graded vector space $E=\oplus_{k \in \mathbb{Z}}E_{k}$.
By Proposition \ref{prop:Otherpaper}, this proves that $\mathcal{N}$ is weak Nijenhuis with respect to $\mu$, which proves $\textit{(d)}$.
\end{proof}
\begin{rem}
 Theorem $2.1.1$ in Delgado \cite{Delgado}, which was our starting point, is a particular case of item $\textit{(b)}$ of Theorem~\ref{Theorem}, by putting $a_i=1$, for all positive integers $i$.
\end{rem}



\subsection*{$L_{\infty}$-structures mixing products and brackets in the presence of  Nijenhuis ${\mathcal C}^\infty$-linear bundle maps on Lie algebroids}
Theorem \ref{Theorem} in particular holds true for the Gestenhaber algebra of a Lie algebroid.

Given a Lie algebroid $(A, \left[.,.\right], \rho)$ over a manifold $M$, we denote
by $ \left[ ., . \right]_{_{SN}}$ the \emph{Schouten-Nijenhuis} bracket on the space of multivectors of $A$ and by $\diff^A$ (or simply $\diff$, if there is no risk of confusion) the differential of $A$.
Set $A'_i:=\Gamma(\wedge^{i+1} A)$ and $ A'=\oplus_{i\geq -1}A'_i$ , with $A'_{-1}=\Gamma(\wedge^0A)=\mathcal{C}^{\infty}(M)$.
It is well known that the Schouten-Nijenhuis bracket is a graded skew-symmetric bracket of degree zero on $A'=\oplus_{i\geq -1}A'_i$ that defines
a graded Lie algebra bracket on $A'=\Gamma(\wedge A)[1]$ and that
the differential $\diff^A$ is a derivation of $\Gamma(\wedge A^*)$ that squares to zero.
 It is also well known that $\left(A'=\Gamma(\wedge A)[1], \left[\cdot, \cdot\right]_{_{SN}}, \wedge\right)$ is a Gerstenhaber algebra.

 \

Given a $\mathcal{C}^{\infty}$-linear bundle map $N$ on a Lie algebroid $(A,\left[\cdot,\cdot\right],\rho)$, we define a linear map $\underline{N}$ on the graded
  vector space $\Gamma(\wedge A)[2]$, by setting
  \begin{equation*}
  \underline{N}(f):=0,
  \end{equation*}
  for all $f\in \mathcal{C}^{\infty}(M)$, and
\begin{equation*}\label{def:ExtensionbyderivationN}
\underline{N}(P):=\sum_{i=1}^{p}P_1\wedge \cdots \wedge P_{i-1}\wedge N(P_i)\wedge P_{i+1}\wedge \cdots \wedge P_p,
\end{equation*}
for all monomial multi-sections $P=P_1\wedge \cdots \wedge P_{p}\in \Gamma(\wedge^p A)[2]$. The map $\underline{N}$ is called the {\em extension of N by derivation}
 on the graded vector space $\Gamma(\wedge A)[2]$.
It is a  derivation on the graded vector space $\Gamma(\wedge A)[2]$, hence a symmetric vector valued $1$-form on $\Gamma(\wedge A)[2]$, and it has degree zero.
\begin{lem}
Let $\left(A , \left[\cdot,\cdot\right], \rho \right)$ be a Lie algebroid, $N:A \rightarrow A $ a $\mathcal{C}^{\infty}$-linear bundle map and $\underline{N}$ its extension by derivation on the space of
multi-sections $\Gamma\left(\wedge A\right)$. Then, for all integers $k\geq 1$, we have
\begin{equation}\label{underlineNmathcalN}
\left[\underline{N}, \mathcal{N}_{k}\right]_{_{RN}}=0,
\end{equation}
where $\mathcal{N}_{k}\left(P_1,\cdots,P_k\right):=P_1\wedge\dots\wedge P_k$, for all homogeneous elements $P_1,\cdots,P_k \in \Gamma(A) $.
\end{lem}
Let us now recall other notions and results that will be used in the sequel.
Let $\left(A, \mu=\left[\cdot,\cdot\right], \rho \right)$ be a Lie algebroid and
$N: A\rightarrow A$ a $\mathcal{C}^{\infty}$-linear bundle map. We may define a deformed bracket
$\mu^N=\left[\cdot,\cdot\right]_{_{N}}$  by setting
 \begin{equation}\label{deformedbracket}
 \left[X,Y\right]_{_{N}}:=\left[NX,Y\right]+\left[X,NY\right]-N\left[X,Y\right].
 \end{equation}
 We denote by $\left[\cdot,\cdot\right]_{_{SN}}^{N}$ the Schouten-Nijenhuis bracket on $\Gamma(\wedge A)$, associated to $\mu^N$. Set
 \begin{equation*}
 l_2^N(P,Q)=(-1)^{p-1}\left[P,Q \right]_{_{SN}}^{N}, \,\,P \in \Gamma(\wedge^p A), Q \in \Gamma(\wedge A).
 \end{equation*}
 Then, we have \cite{AzimiLaurentCosta}
 \begin{equation} \label{deformed_l2}
 [\underline{N},l_2]_{_{RN}}=l_2^N.
 \end{equation}

The Nijenhuis torsion of $N$ with respect to $\mu= [\cdot, \cdot]$, that we denote by $\mathcal{T}_{\mu}N$, is defined by
\begin{equation*}\label{Nijenhuistorsion}
\mathcal{T}_{\mu}N(X,Y):= \left[NX,NY\right]-N\left[X,Y\right]_{_{N}},
\end{equation*}
for all sections $X,Y \in \Gamma \left( A\right)$, or, equivalently, by
\begin{equation}\label{Nijenhuistorsion2}
\mathcal{T}_{\mu}N(X,Y):= \frac{1}{2}\left(\left[X,Y\right]_{_{N,N}}-\left[X,Y\right]_{_{N^2}}\right),
\end{equation}
where $\left[\cdot,\cdot \right]_{_{N,N}}:=\left(\left[\cdot,\cdot \right]_{_{N}}\right)_{_{N}}$ and $N^2=N \circ N$.
If $\mathcal{T}_{\mu}N=0$, then  $N$ is called \emph{Nijenhuis} and $\left(A, \mu^{N}=\left[\cdot,\cdot\right]_{_{N}}, \rho\circ N\right)$ is a Lie algebroid.

\begin{lem}\cite{AzimiLaurentCosta}\label{prop:NijenhuisAsNijenhuis}
For every Nijenhuis $\mathcal{C^{\infty}}$-linear bundle map $N$ on a Lie algebroid $(A,\left[\cdot,\cdot\right],\rho)$,
the extension by derivation of $N$, $\underline N $ , is a Nijenhuis vector valued $1$-form with respect to  the $L_{\infty}$-structure $l_2$, as in Proposition \ref{maintheorem}, on the Gerstenhaber algebra $(\Gamma(\wedge A)[2], \left[\cdot, \cdot \right]_{_{SN}}, \wedge)$, with square ${\underline{N^2}}$, i.e.,
\begin{equation*}
\left[\underline{N}, \left[\underline{N}, l_2\right]_{_{RN}}\right]_{_{RN}}=\left[\underline{N^2}, l_2\right]_{_{RN}} \,\,\mbox{\textrm and} \,\, \,\, \left[\underline{N},\underline{N^2}\right]_{_{RN}}=0.
\end{equation*}
\end{lem}

In the next theorem, we see how a Nijenhuis $\mathcal{C}^{\infty}$-linear bundle map on a Lie algebroid, appears as a Nijenhuis vector valued 1-form with respect to the associated
$L_{\infty}$-algebras, according to Theorem \ref{Theorem}.

\begin{thm}
Let $\left(A, \left[\cdot,\cdot\right], \rho\right)$ be a Lie algebroid and $N$ a Nijenhuis $\mathcal{C^{\infty}}$-linear bundle map. Let $\mu:= \sum_{i\geq 1}a_il_i$ be the  $L_{\infty}$-algebra as in Theorem \ref{Theorem}. Then
$\underline{N}$ is a Nijenhuis vector valued $1$-form with respect to $\mu$, with square $\underline{N^2}$.
\end{thm}
\begin{proof}
Using the graded Jacobi identity of the Richardson-Nijenhuis bracket and  equation (\ref{underlineNmathcalN}) we have
\begin{equation*}\label{underlineN}
\begin{array}{rcl}
\left[\underline{N}, l_k\right]_{_{RN}}&=&\left[\underline{N}, \left[\mathcal{N}_{k-1}, l_2\right]_{_{RN}}\right]_{_{RN}}\\
                                      &=&\left[\left[\underline{N}, \mathcal{N}_{k-1}\right]_{_{RN}}, l_2\right]_{_{RN}}
                                         +\left[\mathcal{N}_{k-1}\left[\underline{N}, l_2\right]_{_{RN}}\right]_{_{RN}}\\
                                      &=& \left[\mathcal{N}_{k-1}\left[\underline{N}, l_2\right]_{_{RN}}\right]_{_{RN}}.
\end{array}
\end{equation*}
Hence, by (\ref{underlineNmathcalN}) and Lemma \ref{prop:NijenhuisAsNijenhuis}, we have

\begin{equation*}
\begin{array}{rcl}
\left[\underline{N},\left[\underline{N}, l_k\right]_{_{RN}}\right]_{_{RN}}&=&\left[\underline{N}, \left[\mathcal{N}_{k-1}\left[\underline{N}, l_2\right]_{_{RN}}\right]_{_{RN}}\right]_{_{RN}}\\
                                                                          &=&\left[ \left[\underline{N}, \mathcal{N}_{k-1}\right]_{_{RN}},\left[\underline{N}, l_2\right]_{_{RN}}\right]_{_{RN}}
                                                                            +\left[\mathcal{N}_{k-1}, \left[\underline{N}, \left[\underline{N}, l_2\right]_{_{RN}}\right]_{_{RN}}\right]_{_{RN}}\\
                                                                          &=&\left[\mathcal{N}_{k-1}, \left[\underline{N}, \left[\underline{N}, l_2\right]_{_{RN}}\right]_{_{RN}}\right]_{_{RN}}\\
                                                                          &=&\left[\mathcal{N}_{k-1}, \left[\underline{N^2}, l_2\right]_{_{RN}}\right]_{_{RN}}\\
                                                                          &=&\left[\left[\mathcal{N}_{k-1},\underline{N^2}\right]_{_{RN}}, l_2\right]_{_{RN}}
                                                                             +\left[\underline{N^2}, \left[\mathcal{N}_{k-1}, l_2\right]_{_{RN}}\right]_{_{RN}}\\
                                                                          &=&\left[\underline{N^2}, \left[\mathcal{N}_{k-1}, l_2\right]_{_{RN}}\right]_{_{RN}}\\
                                                                          &=&\left[\underline{N^2}, l_k\right]_{_{RN}}.
\end{array}
\end{equation*}
Therefore,
\begin{equation*}
\left[\underline{N},\left[\underline{N}, \mu\right]_{_{RN}}\right]_{_{RN}}=\left[\underline{N^2}, \mu\right]_{_{RN}}.
\end{equation*}
Since $\left[\underline{N},\underline{N^2}\right]_{_{RN}}=0$ holds, we conclude that $\underline{N}$ is Nijenhuis with respect to $\mu$ with square $\underline{N^2}$.
\end{proof}

\section{Exact Poisson quasi-Nijenhuis structures with background} \label{section_PqNb}

%
%

In this section we give another example of co-boundary Nijenhuis on an $L_\infty$-algebra, in relation with Poisson structures on Lie algebroids.
Given a Lie algebroid $A$ and a closed $3$-form $H$, we may define an $L_\infty$-structure $\mu$ on the graded vector space
$\Gamma\left(\wedge A\right)\left[2\right]$  obtained out of the Lie algebroid structure and $H$. We study the conditions that a
$\mathcal{C}^{\infty}$-linear bundle map $N:A \to A$, a bivector $\pi$ and a $2$-form $\omega$ on $A$ should satisfy in order that the vector valued form
$\mathcal{N}=\pi+\underline{N}+\underline{\omega}$ is a co-boundary Nijenhuis vector valued form with respect to $\mu$. The conditions we shall obtain are
similar to the conditions which define an exact Poisson quasi-Nijenhuis structure with background on $A$.

  Let $(A, \mu=[\cdot, \cdot],\rho)$ be a Lie algebroid and $\pi\in \Gamma\left(\wedge^{2} A\right)$ a  Poisson bivector, that is, $\left[\pi,\pi \right]_{_{SN}}=0$. The \emph{Koszul bracket}
  $\{\cdot,\cdot\}_{_{\mu}}^{\pi}$ on the space $\Gamma \left(A^{*}\right)$ is defined by:

 \begin{equation}\label{Koszulbracket}
 \{\alpha,\beta\}_{_{\mu}}^{\pi}:=\mathcal{L}_{\pi^{\#}\left(\alpha\right)}\beta-\mathcal{L}_{\pi^{\#}\left(\beta\right)}
 \alpha-\diff^{A}\left(\pi\left(\alpha,\beta\right)\right),
 \end{equation}
 where $\pi^{\#}\left(\alpha\right)$ is the section defined by $\left<\beta, \pi^{\#}\left(\alpha\right)\right>=\pi\left(\alpha, \beta\right)$ and
  $\mathcal{L}_{_{\pi^{\#}\left(\alpha\right)}}\beta$ is the Lie derivative of the form $\beta$ in direction of the section $\pi^{\#}\left(\alpha\right)$.
   It is well known that
   $\left(A^{*},\{\cdot,\cdot\}_{_{\mu}}^{\pi}, \rho \circ \pi^{\#}\right)$ is a Lie algebroid and its differential is given by
   $\diff^{A^{*}}=\left[\pi,\cdot \right]_{_{SN}}$.

   Let $\{\alpha,\beta\}_{_{\mu^{N}}}^{\pi}$ be the Koszul bracket (\ref{Koszulbracket}) on $\Gamma \left(A^{*}\right)$, associated with $\pi$ and $\mu^N=[\cdot,\cdot]_{_{N}}$ (see(\ref{deformedbracket})).
   The Magri-Morosi concomitant $C\left(\pi, N\right)$ is defined, for all $\alpha, \beta \in \Gamma\left(A^{*}\right)$, by
 \begin{equation*}
 C\left(\pi, N\right)\left(\alpha, \beta\right):=\left(\{\alpha, \beta\}_{_{\mu}}^{\pi}\right)_{_{N^{*}}}-\{\alpha,\beta\}_{_{\mu^{N}}}^{\pi},
 \end{equation*}
 where $N^{*}: A^{*}\rightarrow A^{*}$
 is defined by $\left<N^{*} \alpha, X\right>=\left< \alpha, NX\right>$, for all $\alpha \in \Gamma\left(A^{*}\right)$, and $X\in\Gamma\left(A\right)$
 and $\left(\{\cdot, \cdot\}_{_{\mu}}^{\pi}\right)_{_{N^{*}}}$ is the deformed bracket of $\{\cdot,\cdot\}_{_{\mu}}^{\pi}$, according to (\ref{deformedbracket}).

 Given a $2$-form $\omega \in \Gamma(A^*)$ and a $\mathcal{C}^{\infty}$-linear bundle map $N:A \to A$ such that $\omega^{\flat} \circ N =N^{*} \circ \omega^{\flat}$, we denote by $\omega_N$ the $2$-form defined by
 \begin{equation*}\omega_N(X,Y)=\omega(NX,Y)=\omega(X,NY),
 \end{equation*}
for all $X,Y \in \Gamma(A)$.

Let us recall the notion of an exact Poisson quasi-Nijenhuis structure with background on a Lie algebroid.

\begin{defn}\cite{AntunesCosta} \label{def:ExactPQNB}
An exact Poisson
quasi-Nijenhuis structure with background on a Lie algebroid $(A, \left[\cdot,\cdot\right], \rho)$ is a quadruple $(\pi, N, \omega, H)$, where $\pi$
 is a bivector, $N:A \to A$ is a $\mathcal{C}^{\infty}$-linear bundle map, $\omega$ is a 2-form and $H$ is a closed 3-form such that
 $N \circ \pi^{\#} = \pi^{\#} \circ N^{*}$ and $\omega^{\flat} \circ N =N^{*} \circ \omega^{\flat}$ and the following conditions are satisfied:
\begin{enumerate}
 \item [(a)] $\pi$ is Poisson;
 \item [(b)] $C\left(\pi, N\right)\left(\alpha, \beta\right)=2H(\pi^{\#}\alpha, \pi^{\#}\beta, \cdot)$, for all $\alpha, \beta \in \Gamma\left(A^{*}\right)$;
 \item [(c)] $\mathcal{T}_{\mu}N\left(X, Y\right)=\pi^{\#}\left(-H\left(NX, Y, \cdot\right)-H\left(X, NY, \cdot\right)+\diff \omega\left(X, Y, \cdot\right)\right)$,
  for all\\ $X,Y\in \Gamma\left(A\right)$;
 \item [(d)] $\iota_{N}\diff \omega-\diff \omega_{N}-\mathcal{H}= \lambda H$, for some $\lambda \in \mathbb R$, \\
  where $\mathcal{H}\left(X, Y, Z\right)= H\left(NX, NY,Z\right)+ \circlearrowleft_{X,Y,Z}$ and $\circlearrowleft_{X,Y,Z}$ stands for the circular permutation on $X,Y$ and $Z$,
   and  $\iota_{N}\diff \omega\left(X,Y,Z\right)=\diff \omega\left(NX,Y,Z\right)+\diff \omega\left(X,NY,Z\right)+\diff \omega\left(X,Y,NZ\right)$, for all $X, Y, Z\in \Gamma\left(A\right)$.
 \end{enumerate}
\end{defn}


 Similar to the case of a $\mathcal{C}^{\infty}$-linear bundle map,  for  a $k$-form on a Lie algebroid we also consider its extension by derivation \cite{AzimiLaurentCosta}. More precisely, if $\kappa \in\Gamma(\wedge^kA^*)$, the extension of $\kappa$
by derivation is a $k$-linear map, denoted by $\underline{\kappa}$,  given by
\begin{equation*}
\underline{\kappa}(P_1,\cdots,P_k):=\sum_{i_1,\cdots, i_k=1}^{p_1,\cdots, p_k}
(-1)^{\spadesuit}\kappa(P_{1,i_1},\cdots,P_{k,i_k})\widehat{P_{1,i_1}}\wedge\cdots\wedge\widehat{P_{k,i_k}},
\end{equation*}
for all homogeneous multi-sections $P_i=P_{i,1}\wedge\cdots\wedge P_{i,p_i} \in \Gamma(\wedge^{p_i} A)$, with $i=1,\cdots,k,$ where $1\leq i_j\leq p_j$ for
all $1\leq j\leq k$,
\begin{equation*}
\widehat{P_{j,i_j}}=P_{j,1}\wedge\cdots\wedge P_{j,i_j-1}\wedge P_{j,i_j+1}\wedge\cdots\wedge P_{j,p_j}\in \Gamma(\wedge^{p_j-1} A)
\end{equation*}
and
\begin{equation*}
\spadesuit=2p_1+3p_2+\cdots+(k+1)p_k.
\end{equation*}
It follows from its definition that $\underline{\kappa}$  is a derivation on the graded vector space $\Gamma(\wedge A)[2]$  and that it is a symmetric vector
valued $k$-form  of degree $k-2$ on  $\Gamma(\wedge A)[2]$.

The next two lemmas, proved in \cite{AzimiLaurentCosta}, give some results about the Richardson-Nijenhuis bracket of these $\underline{\kappa}$'s.

\begin{lem}\cite{AzimiLaurentCosta}\label{underlinescommut}
Let $(A, \left[\cdot,\cdot\right], \rho)$ be a Lie algebroid, $\alpha \in \Gamma(\wedge^k A^*)$ a $k$-form, $\beta \in \Gamma(\wedge^l A^*)$ an $l$-form, $\omega \in \Gamma(\wedge^2 A^*)$ a $2$-form and $N:A \to A$ is a $\mathcal{C}^{\infty}$-linear bundle map. Then,
\begin{enumerate}
 \item [(a)]
$
\left[\underline{\alpha},\underline{\beta}\right]_{_{RN}}=0;
$
\item [(b)]
$
\left[\underline{N},\underline{\omega}\right]_{_{RN}}=2 \underline{\omega_N}.
$
\end{enumerate}
\end{lem}
\begin{lem}\cite{AzimiLaurentCosta}\label{lem:commutatorAlgebroidContraction}
Let $(A,\left[\cdot,\cdot\right], \rho)$ be a Lie algebroid, with  differential $\diff^A$ and with associated $L_{\infty}$-structure $l_2$, given by
(\ref{definition_l2}), on  $\Gamma(\wedge A)[2]$. Then,
  \begin{equation*}
  \left[\underline{\alpha}, l_2\right]_{_{RN}}= \underline{\diff^A \alpha},
  \end{equation*}
  for all $\alpha \in \Gamma(\wedge^k A^*)$.
\end{lem}

The next lemma establishes some formulas involving the Schouten-Nijenhuis and the Richardson-Nijenhuis brackets, that will be needed in the sequel.
%
%
\begin{lem}\label{lem:condition2}
 Let $\left(A, \left[\cdot,\cdot\right], \rho \right)$ be a Lie algebroid,
 $H$ a  $3$-form, $\pi$ a bivector and $N: A \rightarrow A$ is a ${\mathcal C}^{\infty}$-linear map. Consider the associated $L_\infty$-structure $l_2$ on $\Gamma( \wedge A)[2]$, given by (\ref{definition_l2}). Then, for all $\alpha, \beta \in \Gamma\left(A^{*}\right)$ and
 $X \in \Gamma\left(A\right)$,
 \begin{enumerate}
 \item [(a)] $\underline{N}\left[\pi,X\right]_{_{SN}}\left(\alpha, \beta\right)=\left[\pi,X\right]_{_{SN}}\left(N^{*}\alpha, \beta\right)
 +\left[\pi,X\right]_{_{SN}}\left(\alpha, N^{*}\beta\right),$
 \item [(b)] $\left(\left[\pi, \left[\underline{N}, l_2 \right]_{_{RN}} \right]_{_{RN}}
       + \left[\underline{N}, \left[\pi, l_2 \right]_{_{RN}} \right]_{_{RN}}\right)\left(X\right)\left(\alpha, \beta\right)=
       -C\left(\pi, N\right)\left(\alpha, \beta\right)\left(X\right),$
\item [(c)] $\left[\pi, \left[\pi, \underline{H} \right]_{_{RN}} \right]_{_{RN}}\left(X\right)\left(\alpha, \beta\right)=
2H\left(\pi^{\#}\alpha, \pi^{\#}\beta, X\right).$
       \end{enumerate}
 \end{lem}
 \begin{proof}
 For the sake of simplicity and without loss of generality, assume that $\pi= \pi_1\wedge\pi_2$, with $\pi_1, \pi_2 \in \Gamma(A)$. Notice that
 \begin{equation*}
 \begin{array}{rcl}
 \underline{N}\left[\pi,X\right]_{_{SN}}&=&\underline{N}\left(\left[\pi_1,X\right]\wedge\pi_2-\left[\pi_2,X\right]\wedge\pi_1\right)\\
                                        &=&N\left[\pi_1,X\right]\wedge\pi_2+\left[\pi_1,X\right]\wedge N\pi_2-N\left[\pi_2,X\right]\wedge\pi_1-
                                        \left[\pi_2,X\right]\wedge N\pi_1.
 \end{array}
 \end{equation*}
 Hence, using the fact that  $\left<\alpha, NX\right>=\left<N^{*}\alpha,X\right>$, for all $\alpha\in\Gamma\left(A^{*}\right)  , X\in\Gamma\left(A\right)$
 and after considering suitable terms together, we have
\begin{equation*}
 \begin{array}{rcl}
 \underline{N}\left[\pi,X\right]_{_{SN}}\left(\alpha, \beta\right)
 &=&\left(\left[\pi_1,X\right]\wedge\pi_2-\left[\pi_2,X\right]\wedge N\pi_1\right)\left(N^{*}\alpha, \beta\right)\\
 &&+\left(\left[\pi_1,X\right]\wedge\pi_2-\left[\pi_2,X\right]\wedge N\pi_1\right)\left(\alpha,N^{*} \beta\right).\\
 &=&\left[\pi,X\right]_{_{SN}}\left(N^{*}\alpha, \beta\right)+\left[\pi,X\right]_{_{SN}}\left(\alpha, N^{*}\beta\right).
 \end{array}
 \end{equation*}
This proves $\left(\textit{a}\right)$.
Let $\left[\cdot,\cdot\right]_{_{SN}}^{N}$ be the  Schouten-Nijenhuis bracket associated to the deformed bracket $\mu^{N}$
(see (\ref{deformedbracket})). Then,
 using the definition of the Richardson-Nijenhuis bracket and Equations (\ref{kform}) and (\ref{deformed_l2}), for every section $X\in \Gamma\left(A\right)$ we have
\begin{equation}\label{eq2}
\left(\left[\pi, \left[\underline{N}, l_2 \right]_{_{RN}} \right]_{_{RN}}
       + \left[\underline{N}, \left[\pi, l_2 \right]_{_{RN}} \right]_{_{RN}}\right)\left(X\right)=
        \left[\pi, X\right]_{_{SN}}^{N}+\left[\pi, NX\right]_{_{SN}}-N\left[\pi, X\right]_{_{SN}}.
\end{equation}
Identifying $\left(A^{*}\right)^{*}$
  and $A$, for every $X\in \Gamma\left(A\right)$ and
  $\alpha, \beta \in \Gamma\left(A^{*}\right)$ we obtain:

  \begin{equation}\label{differentialofduble}
  \diff^{A^{*}}\left(X\right)\left(\alpha, \beta\right)=\left[\pi,X\right]_{_{SN}}\left(\alpha, \beta\right).
  \end{equation}
  On the other hand, by definition of $\diff^{A^{*}}$, we have:
  \begin{equation}\label{defidifferentialofduble}
  \diff^{A^{*}}\left(X\right)\left(\alpha, \beta\right)=\rho\circ \pi^{\#}\alpha\left<\beta, X\right>-\rho\circ \pi^{\#}\beta
  \left<\alpha, X\right>-\left<\{\alpha,\beta\}_{_{\mu}}^{\pi}, X\right>.
  \end{equation}
  Equations (\ref{differentialofduble}) and (\ref{defidifferentialofduble}) imply that
  \begin{equation}\label{SCandKoszul}
  \left<\{\alpha,\beta\}_{_{\mu}}^{\pi}, X\right>=-\left[\pi,X\right]_{_{SN}}\left(\alpha, \beta\right)+\rho\circ
   \pi^{\#}\alpha\left<\beta, X\right>-\rho\circ \pi^{\#}\beta\left<\alpha, X\right>.
  \end{equation}
Applying,  successively, (\ref{SCandKoszul}) for $\mu$, $N^{*}\alpha$, $\beta$ and $X$, then for $\mu$, $\alpha$, $N^{*}\beta$ and $X$, and finally
 for $\mu$, $\alpha$, $\beta$ and $NX$, we have:
\begin{equation}\label{NalphabetaX}
  \left<\{N^{*}\alpha,\beta\}_{_{\mu}}^{\pi}, X\right>=-\left[\pi,X\right]_{_{SN}}\left(N^{*}\alpha, \beta\right)
  +\rho\circ \pi^{\#}N^{*}\alpha\left<\beta, X\right>-\rho\circ \pi^{\#}\beta\left<N^{*}\alpha, X\right>,
  \end{equation}
\begin{equation}\label{alphaNbetaX}
  \left<\{\alpha,N^{*}\beta\}_{_{\mu}}^{\pi}, X\right>=-\left[\pi,X\right]_{_{SN}}\left(\alpha, N^{*}\beta\right)
  +\rho\circ \pi^{\#}\alpha\left<N^{*}\beta, X\right>-\rho\circ \pi^{\#}N^{*}\beta\left<\alpha, X\right>,
  \end{equation}
  \begin{equation}\label{alphabetaNX}\begin{array}{rcl}
  \left<N^{*}\{\alpha,\beta\}_{_{\mu}}^{\pi}, X\right>&=&\left<\{\alpha,\beta\}_{_{\mu}}^{\pi}, NX\right>\\
                                                      &=&-\left[\pi,NX\right]_{_{SN}}\left(\alpha, \beta\right)
                                                      +\rho\circ \pi^{\#}\alpha\left<\beta, NX\right>-\rho\circ \pi^{\#}\beta\left<\alpha, NX\right>,
  \end{array}
  \end{equation}
  respectively. Now, applying  (\ref{SCandKoszul}) for $\mu^{N}$, $\alpha$, $\beta$ and $X$ we have:
\begin{equation}\label{muNalphabetaX}
  \left<-\{\alpha,\beta\}_{_{\mu^{N}}}^{\pi}, X\right>=\left[\pi,X\right]_{_{SN}}^{N}\left(\alpha, \beta\right)
  -\rho\circ N\circ \pi^{\#}\alpha\left<\beta, X\right>+\rho\circ N\circ \pi^{\#}\beta\left<\alpha, X\right>.
  \end{equation}
Summing up the equations (\ref{NalphabetaX}),
(\ref{alphaNbetaX}), (\ref{alphabetaNX}) and (\ref{muNalphabetaX}) and using item $\left(\textit{a}\right)$ we get
\begin{equation*}
\left<C\left(\pi, N\right)\left(\alpha, \beta\right), X\right>=
\left(\left[\pi, X\right]_{_{SN}}^{N}+\left[\pi, NX\right]_{_{SN}}-N\left[\pi, X\right]_{_{SN}}\right)\left(\alpha, \beta\right).
\end{equation*}
This together with (\ref{eq2})  prove $\left(\textit{b}\right)$.
 Let $\alpha, \beta\in \Gamma\left(\wedge A^{*}\right)$. Then,

\begin{equation}\label{pisharp}
\pi^{\#}\left(\alpha\right)=\iota_{\alpha}\left(\pi_1\wedge\pi_2\right)=\pi_1\left(\alpha\right)\pi_2-\pi_2\left(\alpha\right)\pi_1
\end{equation}
 and, by definition of $\underline{H}$, we get
\begin{equation*}
\left[\pi, \left[\pi,\underline{H}\right]_{_{RN}} \right]_{_{RN}}\left(X\right)=\underline{H}\left(\pi,\pi,X\right)=2H\left(\pi_1,\pi_2,X\right)\pi_1\wedge \pi_2.
\end{equation*}
Hence,
\begin{equation}\label{LHSofpipiH}
\left[\pi, \left[\pi,\underline{H}\right]_{_{RN}} \right]_{_{RN}}\left(X\right)\left(\alpha, \beta\right)=2H\left(\pi_1,\pi_2, X\right)\pi\left(\alpha, \beta\right).
\end{equation}
On the other hand,
\begin{equation}\label{RHSofpipiH}
\begin{array}{rcl}
2H\left(\pi^{\#}\left(\alpha\right),\pi^{\#}\left(\beta\right), X\right)&=&2H\left(\pi_1\left(\alpha\right)\pi_2-\pi_2
\left(\alpha\right)\pi_1,\pi_1\left(\beta\right)\pi_2-\pi_2\left(\beta\right)\pi_1, X\right)\\
&=&2H\left(\pi_1,\pi_2, X\right)\pi\left(\alpha, \beta\right).
\end{array}
\end{equation}
Equations (\ref{LHSofpipiH}) and (\ref{RHSofpipiH}) prove $\left(\textit{c}\right)$.
 \end{proof}
%
%
\begin{prop}\label{prop:concommitant}
 Let $\left(A , \left[\cdot,\cdot\right], \rho \right)$ be a Lie algebroid,
 $H$ a  $3$-form, $\pi$ a bivector and $N: A \rightarrow A$ is a ${\mathcal C}^{\infty}$-linear map. Consider the associated $L_\infty$-structure $l_2$ on $\Gamma( \wedge A)[2]$, given by (\ref{definition_l2}). Then, for all $\alpha, \beta \in \Gamma\left(A^{*}\right)$ and
  $X \in \Gamma\left(A\right)$,
 \begin{equation}\label{vectorvalued1form}
 \left[\pi, \left[\pi, \underline{H} \right]_{_{RN}} \right]_{_{RN}}
       + \left[\pi, \left[\underline{N}, l_2 \right]_{_{RN}} \right]_{_{RN}}
       + \left[\underline{N}, \left[\pi, l_2 \right]_{_{RN}} \right]_{_{RN}}=0
       \end{equation}
        if and only if
       \begin{equation}\label{concommitantcondition}
       C\left(\pi, N\right)\left(\alpha, \beta\right)=2H\left(\pi^{\#}\alpha, \pi^{\#}\beta, \cdot \right).
       \end{equation}
\end{prop}
\begin{proof}
Assume that (\ref{vectorvalued1form}) holds.  Then, items $\left(\textit{b}\right)$ and
 $\left(\textit{c}\right)$ in Lemma \ref{lem:condition2} imply (\ref{concommitantcondition}).

Assume that (\ref{concommitantcondition}) holds. Set
\begin{equation*}
I:=\left[\pi, \left[\pi, \underline{H} \right]_{_{RN}} \right]_{_{RN}}
       + \left[\pi, \left[\underline{N}, l_2 \right]_{_{RN}} \right]_{_{RN}}
       + \left[\underline{N}, \left[\pi, l_2 \right]_{_{RN}} \right]_{_{RN}}.
\end{equation*}
 Note that $I$ is a vector valued $1$-form of degree $1$. Hence,
if $P= X_1\wedge\cdots\wedge X_n\in \Gamma\left(\wedge^{n} A\right)$ is a $n$-section, then $I(P) \in \Gamma\left(\wedge^{n+1} A\right)$ is a $(n+1)$-section. Since $I$ is a derivation, we have
\begin{equation*}\begin{array}{rcl}
 \iota_{\alpha_1\wedge \cdots\wedge\alpha_{n+1}}I\left(P\right)&=&\sum_{i=1}^{n}\iota_{\alpha_1\wedge \cdots\wedge\alpha_{n+1}}I\left(X_i\right)\wedge {\widehat{P_i}} \\
                                                               &=& \sum_{i=1}^{n}\sum_{\sigma\in Sh\left(2,n-1\right)}I\left(X_i\right)\left(\alpha_{\sigma\left(1\right)}
                                                               ,\alpha_{\sigma\left(2\right)}\right)\iota_{\alpha_{\sigma\left(3\right)}\wedge \cdots\wedge
                                                               \alpha_{\sigma\left(n+1\right)}} {\widehat{P_i}},\\
       \end{array}
       \end{equation*}
       for all $\alpha_1,\cdots, \alpha_{n+1}\in \Gamma\left(\wedge A^{*}\right)$, where ${\widehat{P_i}}=X_1\wedge\cdots\wedge X_{i-1}\wedge X_{i+1}\wedge\cdots\wedge X_{n}$.
 But items $\left(\textit{b}\right)$ and $\left(\textit{c}\right)$ in Lemma \ref{lem:condition2} imply  that
 \begin{equation*}
 I\left(X_i\right)\left(\alpha_{\sigma\left(1\right)},\alpha_{\sigma\left(2\right)}\right)=
 -C\left(\pi, N\right)\left(\alpha_{\sigma\left(1\right)}, \alpha_{\sigma\left(2\right)}\right)\left(X_i\right)
 +2H\left(\pi^{\#}\alpha_{\sigma\left(1\right)}, \pi^{\#}\alpha_{\sigma\left(2\right)}, X_i \right)=0,
 \end{equation*}
 and so, $I(P)=0$, which implies that $I=0$.
\end{proof}
\begin{lem}\label{lem_condition3}
Let $\left(A , \left[\cdot,\cdot\right], \rho \right)$ be a Lie algebroid,
 $H$ a  $3$-form, $\pi$ a bivector and $N: A \rightarrow A$ a ${\mathcal C}^{\infty}$-linear map such that $N\circ \pi^{\#}=\pi^{\#}\circ N^{*} $.
 Then, for all  $X,Y  \in \Gamma\left(A\right)$,
\begin{enumerate}
\item [(a)] $\pi^{\#}\left(H\left(X, Y, \cdot\right)\right)= \underline{H}\left(\pi, X,Y\right)$,
\item [(b)]
 $\left(\left[\pi, \left[\underline{N}, \underline{H} \right]_{_{RN}} \right]_{_{RN}}
       + \left[\underline{N}, \left[\pi, \underline{H} \right]_{_{RN}} \right]_{_{RN}}\right)\left(X, Y\right)=2 \underline{H}\left(\pi, NX, Y\right) + 2 \underline{H}\left(\pi, X, N Y\right)$.
\end{enumerate}
\end{lem}
\begin{proof}
Without loss of generality, we assume that $\pi=\pi_1\wedge \pi_2$, with $ \pi_1, \pi_2 \in \Gamma(A)$. Then, for all $X,Y \in \Gamma\left(A\right)$, using (\ref{pisharp}), we have
\begin{equation*}
\pi^{\#}\left(H\left(X, Y, \cdot\right)\right)=H\left(X, Y, \pi_1\right)\pi_2-H\left(X, Y, \pi_2\right)\pi_1.
\end{equation*}
So, by definition of $\underline{H}$, we have
\begin{equation*}
\pi^{\#}\left(H\left(X, Y, \cdot\right)\right)=\underline{H}\left(X, Y, \pi\right)
\end{equation*}
and, since $\underline{H}$ is graded symmetric, we get $\left(\textit{a}\right)$. A direct computation, using the definition of Richardson-Nijenhuis bracket, shows that for all $X,Y\in \Gamma\left(A\right)$,
\begin{equation}\label{foritemb}
\begin{array}{rcl}
\left(\left[\pi, \left[\underline{N}, \underline{H} \right]_{_{RN}} \right]_{_{RN}}
       + \left[\underline{N}, \left[\pi, \underline{H} \right]_{_{RN}} \right]_{_{RN}}\right)\left(X, Y\right)&=& \\
       2\underline{H}\left(\pi, NX, Y\right)+2\underline{H}\left(\pi, X, NY\right)+\underline{H}\left(\underline{N}\pi, X, Y\right)
       -2\underline{N}\circ\underline{H}\left(\pi, X, Y\right)&&
       \end{array}
\end{equation} holds.
 With a similar argument as in the proof of
 $\left(\textit{b}\right)$ in Lemma \ref{lem:condition2}, it can be shown that for all $\alpha, \beta \in \Gamma\left(A^{*}\right)$,
 \begin{equation*}
 (\underline{N}\pi) \left(\alpha, \beta\right)=\pi\left(N^{*}\alpha, \beta\right)+\pi\left(\alpha, N^{*}\beta\right),
 \end{equation*}
which is equivalent to
 \begin{equation*}
 \left<\left(\underline{N}\pi\right)^{\#}\alpha, \beta\right>=\left<\pi^{\#}(N^{*}\alpha), \beta\right>+\left<\pi^{\#}\alpha, N^{*}\beta\right>.
 \end{equation*}
 Now, $N\circ \pi^{\#}=\pi^{\#}\circ N^{*} $ implies that for all $\alpha, \beta \in \Gamma\left(A^{*}\right)$,
 \begin{equation*}
 \left<\left(\underline{N}\pi\right)^{\#}\alpha, \beta\right>=2\left<N(\pi^{\#}\alpha), \beta\right>,
 \end{equation*}
 which proves that
 \begin{equation*}
 \left(\underline{N}\pi\right)^{\#}=2N \circ \pi^{\#}.
 \end{equation*}
 This implies that $\left(\underline{N}\pi\right)^{\#}\left(H\left(X, Y, \cdot\right)\right)=2N \circ \pi^{\#}\left(H\left(X, Y, \cdot\right)\right).$
 Then, using $\left(\textit{a}\right)$, we get
 \begin{equation*}\label{nottrueformultisections}
 \underline{H}\left(\underline{N}\pi, X, Y\right)-2\underline{N}\circ\underline{H}\left(\pi, X, Y\right)=0.
 \end{equation*}
  This, together with (\ref{foritemb}), proves $\left(\textit{b}\right)$.
\end{proof}
%
%
\begin{prop}\label{prop:torsion}
Let $\left(A , \left[\cdot,\cdot\right], \rho \right)$ be a Lie algebroid,
 $H$ a  $3$-form, $\pi$ a bivector and $N: A \rightarrow A$ a $\mathcal{C}^{\infty}$-linear map such that $N\circ \pi^{\#}=\pi^{\#}\circ N^{*} $. Consider the associated $L_\infty$-structure $l_2$ on $\Gamma( \wedge A)[2]$, given by (\ref{definition_l2}). Then,
 \begin{equation}\label{vectorvalued2formpart}
 \begin{array}{l}
 \left[\pi, \left[\underline{N}, \underline{H} \right]_{_{RN}} \right]_{_{RN}}
       \!\!\!+ \left[\underline{N}, \left[\pi, \underline{H} \right]_{_{RN}} \right]_{_{RN}}
      \!\!\! + \left[\underline{N}, \left[\underline{N}, l_2 \right]_{_{RN}} \right]_{_{RN}}
      \!\!\! + \left[\pi, \underline{\diff \omega} \right]_{_{RN}}+ \left[\underline{\omega}, \left[\pi, l_2 \right]_{_{RN}} \right]_{_{RN}}\\
       =\left[\underline{N^2}, l_2\right]_{_{RN}} + \left[\left[\underline{\omega},\pi \right]_{_{RN}}, l_2 \right]_{_{RN}},
       \end{array}
       \end{equation}
      if and only if
       \begin{equation}\label{torsioninproposition}\mathcal{T}_{\mu}N\left(X, Y\right)=-\pi^{\#}\left(H\left(NX, Y, \cdot\right)+H\left(X, NY, \cdot\right)
       +\diff \omega\left(X, Y, \cdot\right)\right),\end{equation}
  for all $X,Y\in \Gamma\left(A\right)$.
  \end{prop}
  \begin{proof}
  Assume that (\ref{vectorvalued2formpart}) holds.
 Using the graded Jacobi identity of the Richardson-Nijenhuis bracket, we get
 \begin{equation*}
 \left[\pi, \left[\underline{\omega}, l_2\right]_{_{RN}} \right]_{_{RN}} + \left[\left[\underline{\omega}, \pi\right]_{_{RN}}, l_2 \right]_{_{RN}}
 =\left[\underline{\omega}, \left[\pi, l_2\right]_{_{RN}} \right]_{_{RN}}
 \end{equation*}
 or, by Lemma \ref{lem:commutatorAlgebroidContraction},
  \begin{equation*}
 \left[\pi, \underline{\diff \omega} \right]_{_{RN}} + \left[\left[\underline{\omega}, \pi\right]_{_{RN}}, l_2 \right]_{_{RN}}
 =\left[\underline{\omega}, \left[\pi, l_2\right]_{_{RN}} \right]_{_{RN}}.
 \end{equation*}
 Hence, (\ref{vectorvalued2formpart}) is equivalent to
\begin{equation*}
   \begin{array}{rcl}
   \left[\pi, \left[\underline{N}, \underline{H} \right]_{_{RN}} \right]_{_{RN}}
    \!\!   + \left[\underline{N}, \left[\pi, \underline{H} \right]_{_{RN}} \right]_{_{RN}}
    \!\!   + 2\left[\pi, \underline{\diff \omega} \right]_{_{RN}}\!\!\!\!&=&\!\!\left[\underline{N^2}, l_2\right]_{_{RN}}\!\! -  \left[\underline{N}
    , \left[\underline{N}, l_2 \right]_{_{RN}} \right]_{_{RN}}.
       \end{array}
\end{equation*}
In particular, for all sections $X,Y \in \Gamma\left(A\right)$ we have,
\begin{equation}\label{equivalenttotorsionXY}
   \begin{array}{c}
   \left(\left[\pi, \left[\underline{N}, \underline{H} \right]_{_{RN}} \right]_{_{RN}}
    \!\!   + \left[\underline{N}, \left[\pi, \underline{H} \right]_{_{RN}} \right]_{_{RN}}
    \!\!   + 2\left[\pi, \underline{\diff \omega} \right]_{_{RN}}\right)\left(X, Y\right)\\
    =\!\! \left(\left[\underline{N^2}, l_2\right]_{_{RN}}\!\! -  \left[\underline{N}, \left[\underline{N}, l_2 \right]_{_{RN}} \right]_{_{RN}}\right)\left(X, Y\right)
       \end{array}
\end{equation}
which, from item $\left(\textit{b}\right)$ of Lemma \ref{lem_condition3} and Equation (\ref{Nijenhuistorsion2}), is equivalent to
\begin{equation} \label{torsion_with_H_omega}
\underline{H}\left(\pi, NX, Y\right)+\underline{H}\left(\pi, X, NY\right)+\underline{\diff \omega}\left(\pi, X, Y\right)=-\mathcal{T}_{_{\mu}}N\left( X, Y\right).
\end{equation}
By item $\left(\textit{a}\right)$ of Lemma \ref{lem_condition3}, (\ref{torsion_with_H_omega}) is equivalent to (\ref{torsioninproposition}).
   Let $P, Q \in \Gamma (\wedge A)$ be any multi-sections on $A$. Since both sides of Equation (\ref{vectorvalued2formpart}) are derivations in each variable and ${\mathcal C}^\infty$-linear on functions, from the equivalence of (\ref{torsioninproposition}) and (\ref{equivalenttotorsionXY}), we get that
  \begin{equation*}
   \begin{array}{c}
   \left(\left[\pi, \left[\underline{N}, \underline{H} \right]_{_{RN}} \right]_{_{RN}}
    \!\!   + \left[\underline{N}, \left[\pi, \underline{H} \right]_{_{RN}} \right]_{_{RN}}
    \!\!   + 2\left[\pi, \underline{\diff \omega} \right]_{_{RN}}\right)\left(P, Q\right)\\
    =\!\! \left(\left[\underline{N^2}, l_2\right]_{_{RN}}\!\! -  \left[\underline{N}, \left[\underline{N}, l_2 \right]_{_{RN}} \right]_{_{RN}}\right)\left(P, Q\right)
       \end{array}
\end{equation*}
 if and only if Equation (\ref{torsioninproposition}) holds. This proves the proposition.
  \end{proof}
\begin{lem}\label{lem:condition4}
Let $\left(A , \left[\cdot,\cdot\right], \rho \right)$ be a Lie algebroid,
 $H$ a $3$-form, $\pi$ a bivector, $N: A \rightarrow A$ a $\mathcal{C}^{\infty}$-linear map and $\omega$ a $2$-form. Then,
\begin{equation*}
 \left[\underline{N}, \left[\underline{N}, \underline{H}\right]_{_{RN}} \right]_{_{RN}}=
 \left[\underline{N^2}, \underline{H}\right]_{_{RN}} + 2 \mathcal{H} -2 \underline{N}\circ \left[\underline{N}, \underline{H}\right]_{_{RN}},
 \end{equation*}
 where $\mathcal{H}\left(P,Q,R\right)=  \underline{H}\left(\underline{N}P, \underline{N}Q,R\right)+ \circlearrowleft_{P,Q,R}$, for all $P,Q,R\in \Gamma\left(\wedge A\right)$.
\end{lem}
\begin{proof}
Let $P, Q, R \in \Gamma\left(\wedge A\right)$. Then, by definition of the Richardson-Nijenhuis bracket, we have
\begin{equation}\label{NNH}
\begin{array}{rcl}
\left[\underline{N}, \left[\underline{N}, \underline{H}\right]_{_{RN}} \right]_{_{RN}}\left(P, Q, R\right)
&=&\left[\underline{N}, \underline{H}\right]_{_{RN}}\left(\underline{N}P, Q, R\right) + \left[\underline{N}, \underline{H}\right]_{_{RN}}\left(P, \underline{N}Q, R\right)\\
&&+\left[\underline{N}, \underline{H}\right]_{_{RN}}\left(P, Q, \underline{N}R\right)-\underline{N}\circ\left[\underline{N}, \underline{H}\right]_{_{RN}}\left(P, Q, R\right).
\end{array}
\end{equation}
Adding and subtracting $\underline{N^2}\circ \underline{H}\left(P, Q, R\right)$ to the right hand side of (\ref{NNH})
 and collecting suitable terms together we get the desired result.
 \end{proof}
%
 %
 \begin{prop}\label{prop:thelastcondition}
 Let $\left(A , \left[\cdot,\cdot\right], \rho \right)$ be a Lie algebroid,
 $H$ a $3$-form, $\pi$ a bivector, $N: A \rightarrow A$ a $\mathcal{C}^{\infty}$-linear map and $\omega$ a $2$-form. Consider the associated $L_\infty$-structure $l_2$ on $\Gamma( \wedge A)[2]$, given by (\ref{definition_l2}). Then,
 \begin{equation*}
 \begin{array}{c}
       \left[\underline{N}, \left[\underline{N}, \underline{H}\right]_{_{RN}} \right]_{_{RN}}
       + \left[\underline{N}, \left[\underline{\omega}, l_2\right]_{_{RN}}\right]_{_{RN}}
       + \left[\underline{\omega}, \left[\pi, \underline{H}\right]_{_{RN}}\right]_{_{RN}}
       + \left[\underline{\omega}, \left[\underline{N}, l_2\right]_{_{RN}}\right]_{_{RN}}\\
       = \left[\underline{N^2}+ \left[\underline{\omega}, \pi\right]_{_{RN}}, \underline{H}\right]_{_{RN}}
 \end{array}
       \end{equation*}
        if and only if
       \begin{equation}\label{condition_H_N}\mathcal{H}-\underline{N} \circ \left[\underline{N}, \underline{H}\right]_{_{RN}} + \left[\underline{N},
       \underline{\diff \omega}\right]_{_{RN}} -\underline{\diff \omega_N}=0,\end{equation}
 where $\mathcal{H}\left(P,Q,R\right)=  \underline{H}\left(\underline{N}P, \underline{N}Q,R\right)+ \circlearrowleft_{P,Q,R}$, for all $P,Q,R\in \Gamma\left(\wedge A\right)$.
 \end{prop}
 \begin{proof}
 Using the graded Jacobi identity of the Richardson-Nijenhuis bracket and Lemma \ref{lem:commutatorAlgebroidContraction}, we get
 \begin{equation}\label{ForNNHpartB}
 \left[ \underline{\omega}, \left[\underline{N}, l_2\right]_{_{RN}}\right]_{_{RN}}= \left[\underline{N}, \underline{\diff \omega}\right]_{_{RN}}+
 \left[ \left[\underline{\omega}, \underline{N}\right]_{_{RN}}, l_2\right]_{_{RN}}.
 \end{equation}
 From Lemma \ref{underlinescommut} $(a)$ and using again the graded Jacobi identity of the Richardson-Nijenhuis bracket, we obtain
 \begin{equation}\label{ForNNHpartB2}
 \left[ \underline{\omega}, \left[\pi, \underline{H}\right]_{_{RN}}\right]_{_{RN}}=  \left[ \left[\underline{\omega}, \pi\right]_{_{RN}}
 \underline{H}\right]_{_{RN}}.
 \end{equation}
 Equations (\ref{ForNNHpartB}) and (\ref{ForNNHpartB2}) together with Lemmas \ref{lem:condition4} and \ref{underlinescommut} $(b)$ imply that
\begin{equation*}\begin{array}{c}\left[\underline{N}, \left[\underline{N}, \underline{H}\right]_{_{RN}} \right]_{_{RN}}
    + \left[\underline{N}, \left[\underline{\omega}, l_2\right]_{_{RN}}\right]_{_{RN}}
    + \left[\underline{\omega}, \left[\pi, \underline{H}\right]_{_{RN}}\right]_{_{RN}}
       + \left[\underline{\omega}, \left[\underline{N}, l_2\right]_{_{RN}}\right]_{_{RN}}\\
       = \left[\underline{N^2}+ \left[\underline{\omega}, \pi\right]_{_{RN}}, \underline{H}\right]_{_{RN}},\end{array}\end{equation*} if and only if
\begin{equation*}\mathcal{H}-\underline{N} \circ \left[\underline{N}, \underline{H}\right]_{_{RN}} + \left[\underline{N}, \underline{\diff \omega}\right]_{_{RN}}-\underline{\diff \omega_N }=0.
\end{equation*}
\end{proof}
The next corollary is a consequence of the fact that \begin{equation}\label{condition_H_N_XYZ}\underline{N} \circ \left[\underline{N}, \underline{H}\right]_{_{RN}}\left(X, Y, Z\right)=0\end{equation}
and $\left[\underline{N}, \underline{\diff \omega}\right]_{_{RN}}(X,Y,Z)=\iota_{N}\diff \omega(X,Y,Z)$, for all
 $X, Y, Z\in \Gamma\left(A\right).$
\begin{cor}\label{cor:thelastcondition}
Let $\left(A , \left[\cdot,\cdot\right], \rho \right)$ be a Lie algebroid,
 $H$ a $3$-form, $\pi$ a bivector, $N: A \rightarrow A$ a $\mathcal{C}^{\infty}$-linear map and $\omega$ a $2$-form. Then,
 \begin{equation}\label{condition_H_N_omega}
 \begin{array}{c}
       \left[\underline{N}, \left[\underline{N}, \underline{H}\right]_{_{RN}} \right]_{_{RN}}
       + \left[\underline{N}, \left[\underline{\omega}, l_2\right]_{_{RN}}\right]_{_{RN}}
       + \left[\underline{\omega}, \left[\pi, \underline{H}\right]_{_{RN}}\right]_{_{RN}}
       + \left[\underline{\omega}, \left[\underline{N}, l_2\right]_{_{RN}}\right]_{_{RN}}\\
       = \left[\underline{N^2}+ \left[\underline{\omega}, \pi\right]_{_{RN}}, \underline{H}\right]_{_{RN}}
       \end{array}
       \end{equation}
     if and only if
       \begin{equation}\label{condition_N_omega}\left(\mathcal{H}+\iota_{N}\diff \omega-\diff \omega_N \right)\left(X, Y, Z\right)=0,\end{equation}
  for all $X,Y,Z\in \Gamma\left(A\right)$, where $\mathcal{H}\left(X, Y, Z\right)= H\left(NX, NY,Z\right)+ \circlearrowleft_{X,Y,Z}$.
\end{cor}
\begin{proof}
Taking into account Equation (\ref{condition_H_N_XYZ}), condition (\ref{condition_N_omega}) implies that Equation (\ref{condition_H_N}) holds, when restricted to sections of $A$. So, by Proposition \ref{prop:thelastcondition},  (\ref{condition_N_omega}) is equivalent to Equation (\ref{condition_H_N_omega}), when restricted to sections of $A$. Since both sides of (\ref{condition_H_N_omega}) are derivations in each variable and $C^\infty$-linear on functions, the result follows.
\end{proof}
\begin{prop}\label{2ndlemma}
 Let $\left(A , \left[\cdot,\cdot\right], \rho \right)$ be a Lie algebroid and
 $H\in \Gamma\left(\wedge^{3} A^{*}\right)$ a closed $3$-form. Then $\mu:=l_2+\underline{H}$ is an $L_{\infty}$-algebra structure on $\Gamma\left(\wedge A\right)[2]$.
\end{prop}
\begin{proof}
 Observe that $l_2$ and $\underline{H}$ are both graded symmetric vector valued forms of degree $1$. Using Lemmas \ref{underlinescommut} and
 \ref{lem:commutatorAlgebroidContraction}, we have
\begin{equation*}\label{[l_2,H]}
\left[l_2+\underline{H}, l_2+\underline{H}\right]_{_{RN}}= \left[l_2, l_2\right]_{_{RN}}+2\left[l_2,
 \underline{H}\right]_{_{RN}}+\left[\underline{H}, \underline{H}\right]_{_{RN}}=0.
\end{equation*}
Therefore, by Theorem \ref{th:linftyRN}, $\mu$ is an $L_{\infty}$-algebra structure on $\Gamma\left(\wedge A\right)[2]$.
\end{proof}

In the next theorem we give necessary and sufficient conditions for a vector valued form $\mathcal{N}:= \pi + \underline{N} + \underline{\omega}$
 to be co-boundary Nijenhuis with respect to $\mu=l_2 + \underline{H}$,
where $\pi$ is a bivector, $N: A \rightarrow A$ is a $\mathcal{C}^{\infty}$-linear map, $\omega$ is a $2$-form and $H$ is a closed $3$-form.
%
%
\begin{thm}\label{thm:maintheorem}
Let $\left(A , \left[\cdot,\cdot\right], \rho \right)$ be a Lie algebroid, $H\in \Gamma\left(\wedge^{3}A^{*}\right)$ a closed $3$-form, $\pi\in \Gamma\left(\wedge^{2}A\right)$, $\omega\in \Gamma\left(\wedge^{2}A^{*}\right)$,
$N:A \rightarrow A$ a $\mathcal{C}^{\infty}$-linear bundle map such that $N\circ\pi^{\#}=\pi^{\#}\circ N^{*}$ and $\omega^{\flat} \circ N =N^{*} \circ \omega^{\flat}$.
Then, $\mathcal{N}=\pi+\underline{N}+\underline{\omega}$ is a co-boundary Nijenhuis vector valued form with respect to the
 $L_{\infty}$-algebra  $\left(\Gamma\left(\wedge A\right)[2], \mu=l_2 + \underline{H}\right)$ with square
 $\underline{N^2}+\left[\underline{\omega}, \pi \right]_{_{RN}}$, if and only if the quadruple $(\pi, N, -\omega, H)$ is an exact Poisson quasi-Nijenhuis structure with background, in the sense of Definition \ref{def:ExactPQNB}, with coefficient $\lambda=0$.
\end{thm}
\begin{proof} Using Lemmas \ref{lem:commutatorAlgebroidContraction} and \ref{underlinescommut}, we have
\begin{equation*}\label{[N mu]}
   \begin{array}{rcl}
     \left[\mathcal{N}, \mu\right]_{_{RN}} &=& \left[\pi, l_2 \right]_{_{RN}} + \left[\pi,\underline{H}\right]_{_{RN}} +\left[\underline{N}
                                                , l_2\right]_{_{RN}} +\left[\underline{N}, \underline{H}\right]_{_{RN}}+\underline{\diff \omega}.\\
   \end{array}
\end{equation*}
Let\\
\begin{equation*}\left[\mathcal{N}, \left[\mathcal{N}, \mu\right]_{_{RN}}\right]_{_{RN}}^{0}=l_2\left(\pi, \pi \right),\end{equation*}
\begin{equation*}\left[\mathcal{N}, \left[\mathcal{N}, \mu\right]_{_{RN}}\right]_{_{RN}}^{1}=\left[\pi, \left[\pi, \underline{H} \right]_{_{RN}} \right]_{_{RN}}
       + \left[\pi, \left[\underline{N}, l_2 \right]_{_{RN}} \right]_{_{RN}}
       + \left[\underline{N}, \left[\pi, l_2 \right]_{_{RN}} \right]_{_{RN}},\end{equation*}
       \begin{equation*}
       \begin{array}{rcl}
       \left[\mathcal{N}, \left[\mathcal{N}, \mu\right]_{_{RN}}\right]_{_{RN}}^{2}&=&
\left[\pi, \left[\underline{N}, \underline{H} \right]_{_{RN}} \right]_{_{RN}}
       + \left[\underline{N}, \left[\pi, \underline{H} \right]_{_{RN}} \right]_{_{RN}}
       + \left[\underline{N}, \left[\underline{N}, l_2 \right]_{_{RN}} \right]_{_{RN}}\\
       &&+ \left[\pi, \underline{\diff \omega} \right]_{_{RN}}+ \left[\underline{\omega}, \left[\pi, l_2 \right]_{_{RN}} \right]_{_{RN}}
       \end{array}
       \end{equation*}
       and
\begin{equation*}
       \begin{array}{rcl}
\left[\mathcal{N}, \left[\mathcal{N}, \mu\right]_{_{RN}}\right]_{_{RN}}^{3}&=&
 \left[\underline{N}, \left[\underline{N}, \underline{H}\right]_{_{RN}} \right]_{_{RN}}
       + \left[\underline{N}, \left[\underline{\omega}, l_2\right]_{_{RN}}\right]_{_{RN}}
       + \left[\underline{\omega}, \left[\pi, \underline{H}\right]_{_{RN}}\right]_{_{RN}}\\
       &&+ \left[\underline{\omega}, \left[\underline{N}, l_2\right]_{_{RN}}\right]_{_{RN}},
       \end{array}
       \end{equation*}
and
\begin{equation*}\left[\mathcal{N}, \left[\mathcal{N}, \mu\right]_{_{RN}}\right]_{_{RN}}^{4}=\left[\underline{\omega}, \left[\underline{N}
, \underline{H}\right]_{_{RN}}\right]_{_{RN}}=0.\end{equation*}
Then,
\begin{equation*}\label{alltogether}
\left[\mathcal{N}, \left[\mathcal{N}, \mu\right]_{_{RN}}\right]_{_{RN}}=\sum_{i=0}^{4}\left[\mathcal{N}, \left[\mathcal{N}, \mu\right]_{_{RN}}\right]_{_{RN}}^{i}.
\end{equation*}
By construction, each $\left[\mathcal{N}, \left[\mathcal{N}, \mu\right]_{_{RN}}\right]_{_{RN}}^{i}$ is a vector valued $i$-form and
 \begin{equation*}\left[\underline{N^2}+\left[\underline{\omega}, \pi \right]_{_{RN}}, \mu \right]_{_{RN}}\end{equation*}
 is the sum of a vector valued $2$-form and a vector valued $3$-form. Therefore,
\begin{equation*}
\left[\mathcal{N}, \left[\mathcal{N}, \mu\right]_{_{RN}}\right]_{_{RN}}=
\left[\underline{N^2}+\left[\underline{\omega}, \pi \right]_{_{RN}}, \mu \right]_{_{RN}},
\end{equation*}
holds, if and only if
\begin{equation*}
\left[\mathcal{N}, \left[\mathcal{N}, \mu\right]_{_{RN}}\right]_{_{RN}}^{0}=0,
\end{equation*}
\begin{equation*}
\left[\mathcal{N}, \left[\mathcal{N}, \mu\right]_{_{RN}}\right]_{_{RN}}^{1}=0,
\end{equation*}
\begin{equation*}
\left[\mathcal{N}, \left[\mathcal{N}, \mu\right]_{_{RN}}\right]_{_{RN}}^{2}=
\left[\underline{N^2}+\left[ \underline{\omega}, \pi \right]_{_{RN}}, l_2 \right]_{_{RN}},
\end{equation*}
\begin{equation*}
\left[\mathcal{N}, \left[\mathcal{N}, \mu\right]_{_{RN}}\right]_{_{RN}}^{3}=
\left[\underline{N^2}+\left[\underline{\omega}, \pi \right]_{_{RN}}, \underline{H} \right]_{_{RN}}.
\end{equation*}
Hence, using Propositions \ref{prop:concommitant}, \ref{prop:torsion} and Corollary \ref{cor:thelastcondition},  $\mathcal{N}$ is co-boundary Nijenhuis
 with respect to $\mu=l_2+\underline{H}$, if and only if
\begin{enumerate}
 \item [(a)] $l_2\left(\pi, \pi\right)=0$, i.e. $\pi$ is Poisson,
 \item [(b)] $C\left(\pi, N\right)\left(\alpha, \beta\right)=2H(\pi^{\#}\alpha, \pi^{\#}\beta, \cdot)$, for all $\alpha, \beta \in \Gamma\left(A^{*}\right)$,
 \item [(c)] ${\mathcal T}_{\mu}N\left(X, Y\right)=-\pi^{\#}\left(H\left(NX, Y, \cdot\right)+H\left(X, NY, \cdot\right)+\diff \omega\left(X, Y, \cdot\right)\right)$,\\
  for all $X,Y\in \Gamma\left(A\right)$,
 \item [(d)] $\left(\mathcal{H}+\iota_{N}\diff \omega -\diff \omega_N\right)(X,Y,Z)=0$, for all $X,Y,Z \in \Gamma\left(A\right)$.
 \end{enumerate}
According to Definition \ref{def:ExactPQNB}, items $(a)-(d)$ mean that the quadruple $(\pi, N, -\omega, H)$ is an exact Poisson quasi-Nijenhuis structure with background with $\lambda=0$. This proves the theorem.
\end{proof}

The next proposition establishes a relation between co-boundary Nijenhuis vector valued forms and Poisson quasi-Nijenhuis structures on manifolds in the sense of Sti\'enon and Xu \cite{Stienon&Xu}.
\begin{prop}
Let $M$ be a manifold, $\pi$ a bivector field,
$N:TM \rightarrow TM$ a
$(1,1)$-tensor and $\omega$ a $2$-form such that $N\circ\pi^{\#}=\pi^{\#}\circ N^{*}$ and $\omega^{\flat} \circ N =N^{*} \circ \omega^{\flat}$.
Then there exist, around each point of $M$, a $2$-form $\alpha$ such that $\mathcal{N}=\pi+\underline{N}+\underline{\omega}$ is a co-boundary Nijenhuis vector valued form with respect to the
 $L_{\infty}$-algebra  $\left(\Gamma\left(\wedge TM\right)[2], \mu=l_2 \right)$ with square
 $\underline{N^2}+\left[ \underline{\omega}, \pi \right]_{_{RN}}+\underline{\alpha}$,
 if and only if, the triple $\left(\pi, N, -\diff\omega\right)$ is a Poisson quasi-Nijenhuis structure on $M$
 in the sense of Sti\'enon and Xu.
\end{prop}

\begin{proof}
Recall that Sti\'enon and Xu in  \cite{Stienon&Xu} defined Poisson quasi-Nijenhuis structures on manifolds as triples $(\pi, N, \phi=-\diff \omega)$ such that
\begin{enumerate}
 \item [(a)] $\pi$ is Poisson,
 \item [(b)] $C\left(\pi, N\right)\left(\alpha, \beta\right)=0$,
 \item [(c)] $T_{\mu}N\left(X, Y\right)=-\pi^{\#}\left(\diff \omega\left(X, Y, \cdot\right)\right)$,\
  for all $X,Y\in \Gamma\left(A\right)$,
 \item [(d)] $\diff \left(\iota_N \diff \omega\right)=0$.
 \end{enumerate}
A direct computation, done as in the proof of Theorem \ref{thm:maintheorem} for $H=0$, gives that  $\mathcal{N}=\pi+\underline{N}+\underline{\omega}$ is a co-boundary Nijenhuis vector valued form with respect to the
 $L_{\infty}$-algebra  $\left(\Gamma\left(\wedge TM\right)[2], \mu=l_2 \right)$ with square
 $\underline{N^2}+\left[ \underline{\omega}, \pi \right]_{_{RN}}+\underline{\alpha}$, if and only if (a), (b) and (c) hold together with the condition  $\left(\iota_{N}\diff \omega -\diff \omega_N - \diff \alpha \right)(X,Y,Z)=0$, for all vector fields $X,Y$ and $Z$ on $M$. Then, by Poincar\'e Lemma, such an $\alpha$ exists around each point of $M$, if and only if $\diff \left(\iota_N \diff \omega\right)=0$.
\end{proof}

\noindent {\bf Acknowledgments.} The authors acknowledge C. Blohmann for sending us manuscript of Delgado \cite{Delgado} which was the starting point of this study. They are also grateful to N. L. Delgado and P. Antunes for their collaboration.

 M. J. Azimi and J. M. Nunes da Costa acknowledge the support of the Centre for Mathematics
of the University of Coimbra - UID/MAT/00324/2013, funded by the Portuguese
Government through FCT/MCTES and co-funded by the European Regional
Development Fund through the Partnership Agreement PT2020. M. J. Azimi acknowledges the support of FCT grant BPD/97717/2013.

 \end{document}